\documentclass[11pt]{amsart}
\usepackage[top=1in, bottom=1.25in, left=1.25in, right=1.25in]{geometry}
\usepackage{amssymb,amsmath,amsthm,mathrsfs}
\usepackage{etoolbox}
\usepackage{hyperref}

\usepackage{url}
\usepackage{cite}
\usepackage{amsfonts}
\usepackage[dvipsnames]{xcolor}
\usepackage{verbatim}
\usepackage{graphicx}

\newtheorem{thm}{Theorem}
\newtheorem{defn}[thm]{Definition}
\newtheorem{prop}[thm]{Proposition}
\newtheorem{cor}[thm]{Corollary}
\newtheorem{lem}[thm]{Lemma}

\theoremstyle{definition}
\newtheorem{remark}[thm]{Remark}

\def\R{{\mathbb{R}}}

\def\P{{\mathcal{P}}}

\def\Lie{{\text{\rm{Lie}}}}
\def\inj{{\text{\rm{inj}}}}
\def\vol{{\text{\rm{vol}}}}

\def\dist{{\text{\rm{dist}}}}
\def\Ad{{\text{\rm{Ad}}}}
\newcounter{constE}
\renewcommand{\theconstE}{{{C}_{\arabic{constE}}}}
\newcommand{\constE}{\refstepcounter{constE}\theconstE}

\title[An avoidance principle and Margulis functions]{An avoidance principle and Margulis functions for expanding translates of unipotent orbits}
\author{Anthony Sanchez}
\address{Department of Mathematics, University of California San Diego,\\ 9500 Gilman Dr, La Jolla,
CA 92093, USA}
\email{ans032@ucsd.edu}

\author{Juno Seong}
\address{Department of Mathematics, University of California San Diego,\\ 9500 Gilman Dr, La Jolla,
CA 92093, USA}
\email{jseong@ucsd.edu}


\begin{document}

\maketitle 

\begin{abstract}
We prove an avoidance principle for expanding translates of unipotent orbits for some semisimple homogeneous spaces. In addition, we prove a quantitative isolation result of closed orbits and give an upper bound on the number of closed orbits of bounded volume. The proofs of our results rely on the construction of a Margulis function and the theory of finite dimensional representations of semisimple Lie groups.
\end{abstract}

\section{Introduction} 

Avoidance principles ---  quantifying how much time trajectories avoid certain subsets of the ambient space --- have been fruitful in the study of dynamical systems. An important example is the non-divergence of unipotent flows which goes back to Margulis \cite{MR0291352}. A quantitative version of non-divergence appears in Dani \cite{MR530631} and were key in Ratner's seminal theorems on unipotent flows \cite{MR1062971,MR1075042,MR1135878,MR1106945}.

Two successful strategies to prove such avoidance principles has been through the construction of \emph{Margulis functions} which originated in the influential work of Eskin--Margulis--Mozes \cite{MR1609447} and the linearization technique of Dani--Margulis \cite{MR1237827}. 

 The flexibility offered by the construction of a Margulis function make them applicable to settings where unipotent dynamics are not available or poorly understood. For example, they appear in the important work of Benoist-Quint \cite{MR2831114,MR3037785,MR3092475} and the recent generalizations of Eskin--Lindenstrauss \cite{EL_short, EL_long} on stationary measures of homogeneous spaces. Additionally, Margulis functions are utilized in Eskin--Mirzakhani--Mohammadi \cite{MR3418528} to prove an avoidance principle that was crucially used to show an analog of Ratner's orbit closure theorem. 
 
 We highlight some other examples to indicate the breadth of Margulis functions, but we recommend the wonderful survey of Eskin--Mozes \cite{EskinMozes} for a more complete overview of the literature. Margulis functions appear: in the setting of Teichm{\"u}ller dynamics by Eskin--Masur \cite{MR1827113} and Athreya \cite{MR2247652}, in the space of lattices by Kadyrov-- Kleinbock--Lindenstrauss--Margulis \cite{MR3736492} and Kleinbock-- Mirzadeh \cite{KM}, for infinite homogeneous spaces by Mohammadi--Oh \cite{Isolations}, and in the space of closed subgroups of a semisimple Lie group equipped with the Chabauty topology in the work of  Gelander--Levit--Margulis \cite{GLM} and Fraczyk--Gelander \cite{FG}.

We use Margulis functions and the theory of finite dimensional representations of semisimple Lie groups to prove an avoidance principle. Broadly speaking, our results rely on the hyperbolicity of diagonal actions and the fact that the perturbation by a foliation often places one in a general position where one expects expansion by the diagonal direction.  

Throughout this paper, $G$ will be a semisimple algebraic Lie group without compact factors and $H$ will be a semisimple subgroup of $G$ without compact factors such that $C_G(H)$ is finite. We let $X:=G/\Gamma$ where $\Gamma$ is a lattice.

We equip $\Lie(G)$ with an inner product that induces a Riemannian metric on $G$. The notion of distances and volumes makes sense with respect to this Riemanninan metric. Denote by by $\inj(x)$ the injectivity radius at point $x$. See the next section for formal descriptions of these notions.

\begin{defn} 
For a pair of positive real numbers $(V, d)$, we say that a point $x\in X$ is \emph{$(V, d)$-Diophantine with respect to $H$} if the following holds: for any intermediate subgroup $H\subseteq S\subsetneq G$ and any closed $S$-orbit $Y=Sx'$ with $\vol (Y) \le V$, the distance between $x$ and $Y$ is at least $d$; namely, $\dist(x,Y) \ge d$. 

For $r>0$, if a point $x\in X$ is $(V, d)$-Diophantine with respect to $H$ and $\inj(x) \ge r$, then we say that $x$ is \emph{$(V, d, r)$-Diophantine with respect to $H$}.
\end{defn}

We fix a one parameter subgroup of diagonalizable elements $\{ a_{t} \} \subseteq H$ and let $U$ be the unstable horospherical subgroup with respect to $\{ a_{t} \}$;
$$U = \{ u\in G : a_{t}ua_{-t} \rightarrow e  \ as \ t \  \rightarrow -\infty \}.$$

We work with the operators
$$(A_{r,t}f)(x) = \frac{1}{m_U(B_r^U)}\int_{B_r ^U}f(a_{t}u x)\,dm_U(u),$$
where $B_{r}^{U} $ is the ball of radius $r$ in $U$ and $m_U$ is the Haar measure on the Lie subgroup $U$ normalized so that $B_1 ^U$ has measure 1. Here the implicit metric on $U$ comes from the identification of $\Lie(U)$ with a Euclidean space. See the next section for details. When considering $A_{1,t}$, we use the notation $A_t$.

We use the operators $A_{r,t}$ to prove a result on the behavior of points of the form $a_tux$ for $u\in B_1 ^U$ and large $t>0$. The following is our main theorem. 

\begin{thm}[Avoidance Principle]\label{MainThm}
Let $G$ be a semisimple group without compact factors and $H$ be a semisimple subgroup without compact factors such that $C_G(H)$ is finite. Let $X=G/\Gamma$ where $\Gamma$ is a lattice. There exists absolute constants $D=D(\dim(G))>0$, $A=A(G/\Gamma, H)>0$, and $C=C(G/\Gamma, H)>0$ such that the following dichotomy holds: for any $x\in X$, there exists $T_x >0$ such that for any pair of $T>T_{x}$ and $R>2$, either:
\begin{itemize}
    \item[(1)] $x$ is not $(R,1/T)$-Diophantine or
    \item[(2)] for all $t \ge A\log T$, \begin{align*}
    & m_{U}\left(\left\{u\in B_1 ^U: a_{t}ux \text{\rm{ is not }\it} 
    (R, R^{-D}, R^{-D}) \text{\rm{-Diophantine with respect to }\it} H\right\}\right) \\
    & < CR^{-1}.
    \end{align*}
\end{itemize}
Moreover, if $K \subseteq X$ is compact, then, $T_x$ can be chosen to be uniform over all $x \in K$.
\end{thm}

\begin{remark} We note a generalization of our main result for solvable epimorphic subgroups. Recall, a subgroup $G'$ of a real algebraic group $G$ is called epimorphic in $G$ if any $G'$-fixed vector is also $G$-fixed for any finite dimensional algebraic representation of $G$. Proposition 2.2. of Shah and Weiss \cite{MR1756987} gives an analogous result to our Linear Algebra Lemma (Lemma \ref{Lem:General_LA_lemma}) for solvable epimorphic groups. Hence, it is plausible that our result can be further generalized so that $B_1^U$ in condition (2) is replaced by $B_1^N$ where $N \subseteq U$ is an algebraic unipotent subgroup normalized by $\{a_t\}$, such that the subgroup generated by $\{a_t\}$ and $N$ is solvable and epimorphic in $G$. 
\end{remark}

The result is similar to Lindenstrauss--Margulis--Mohammadi-Shah \cite{LMMS} though they are interested in avoidance principles for unipotent flows and work in a more general setting. It is also similar to the work of B\'{e}nard--de Saxc\'{e}\cite{BS}.

To prove our main result, we need results on the quantitative isolation of closed orbits which are interesting in their own right. The following theorem is analogous to Lemma 10.3.1 of Einsiedler--Margulis--Venkatesh \cite{MR2507639}.

\begin{thm}[Quantitative Isolation of closed orbits]\label{Thm:Isolation_in_distance} There exists a global constant $D = D(\dim(G)) >0$ such that the following holds: for all intermediate subgroup $H \subseteq S \subsetneq G$ and closed $S$-orbits $Y=Sy$ and $Z=Sz$ of finite volume,
$$\dist(Y \cap K, Z) \gg_K \vol(Y)^{-D}\vol(Z)^{-D}$$
where $K$ is a compact subset of $X$.
\end{thm}
We note that the proof of \cite[Lemma 10.3.1]{MR2507639}, relies on uniform spectral gap for periodic $S$-orbits ($H \subseteq S \subsetneq G$) in congruence quotients. Our proof is arguably softer. In particular, it does not require $\Gamma$ to be arithmetic. The main idea of the proof is to estimate the size of the additive constant of a Margulis function, and goes back to Margulis' unpublished notes (see also Theorem 1.1. of \cite{Isolations}).

Using Theorem \ref{Thm:Isolation_in_distance} above, an upper bound can be obtained on the number of closed orbits of bounded volume. The theorem below is analogous to Corollary 10.7 of Mohammadi--Oh \cite{Isolations}.

\begin{thm}[Upper bound for the number of closed orbits of bounded volume]\label{finiteness_closed_orbits_bounded_volume} There exists a global constant $D = D(\dim(G)) \gg 1$ such that for any 
intermediate subgroup $H\subseteq S\subsetneq G$,
$$\#\{Y: Y=Sy\text{\rm{ is a closed }\it}S\text{\rm{-orbit and }\it}\vol(Y)\le R\} \ll R^{D}.$$
\end{thm}


\subsection{Acknowledgements}
The authors would like to generously thank Amir Mohammadi for suggesting this line of work and for patiently answering many technical questions. We are also grateful to Asaf Katz for helpful discussions regarding his work \cite{katz/online}. 
A.S. was supported by the National Science Foundation Postdoctoral Fellowship under grant number DMS-2103136.

\section{Preliminaries}
In this section we fix notation.

Equip $\Lie(G)$ with the Killing form. This induces
\begin{enumerate}
    \item a norm $\|\cdot\|$ on $\Lie(G)$.
    \item a right-invariant Riemannian metric on $G$ that induces a right-invariant metric on $G$ denoted as $\dist_G$. 
    \item a  metric on $X=G/\Gamma$ denoted as $\dist$ so that the canonical projection $G\to X$ is a local isometry. 
    \item a volume for a closed orbit $H$-orbit on $X$ induced from the Riemannian structure on $G$ which we denote with $\vol$.
\end{enumerate}

With respect to the norm $\|\cdot\|$ on $\Lie(G)$, we can define the unit ball in $\Lie(G)$ which we denote as $B_1 ^{\Lie(G)}$. 

We choose a inner product on $\Lie(U)$ that comes from the identification of $\Lie(U)$ to $\R ^ {d_U}$ where $d_U$ denotes the dimension of $\Lie(U)$. For any $\eta>0$, we can use the inner product on $Lie(U)$ to define a norm (resp. metric) on $\Lie(U)$ (resp. $U$). This allows us to make sense of the unit ball in $\Lie(U)$ which we denote as $B_1 ^{\Lie(U)}$ (resp. in $U$ which we denote as $B_1 ^U$).

For each $x \in X$, we denote by $\inj(x)$ the injectivity radius at point $x$; the supremum of all $\eta>0$ for which the projection map $g \rightarrow gx$ from $G$ to $X=G/\Gamma$ is injective on $B_{\eta}^G$. In Section \ref{Return lemma and Number of Nearby Sheets}, we shall choose a specific $\varepsilon_X>0$ and denote $X_{\varepsilon_X} := \{ x \in X : \inj(x) \ge \varepsilon_X \}$  as the compact part of $X$. Since the exponential map $\Lie(G) \rightarrow G $ defines a local diffeomorphism, there exists an absolute constant $ \sigma_{0} >1 $ such that for all $ w \in \Lie(G)$ with $||w|| \le \epsilon_{X}$ and $x \in X_{\varepsilon_X}$,
    $$ \sigma_{0}^{-1} ||w|| \le \dist(x, \exp(w)x) \le \sigma_{0}||w||.$$ 
By noting that the canonical projection $G\to X$ is a local isometry, we have a way of locally measuring distances in $X$ with the norm on $\Lie(G)$.

For any intermediate subgroup $H \subseteq S
\subseteq G$, we denote the dimension of $\Lie(S)$ by $\dim(S)$ or simply, $d_S$. 

We will denote the Haar measure on $G$ by $m_G$. For the horospherical subgroup $U$ of $G$, we denote the Haar measure on $U$ by $m_U$. 

Let $T$ denote a maximal Cartan subgroup containing $(a_t)_{t\in\mathbb R}$. Let $\rho:G\to GL(V)$ be a finite dimensional representation. Let $\Phi$ denote the root system of $\Lie(G)$ and decompose the vector space into weight spaces $V=\oplus_{\beta\in \Phi} V_\beta$ where $$V_\beta =\{v\in V:\rho(\tau)v=\exp(\beta(\log(\tau)))v,\forall \tau \in T\}$$ is the weight space with weight $\beta\in \Phi$. Choose a basis $(v_{\beta,i})_{i=1} ^{\text{dim }V_\beta}$ so that every $v\in V$ can be written in the form
$$v=\sum_{\beta \in \Phi}\sum_{i=1} ^{\text{dim }V_\beta}c_{\beta,i}v_{\beta,i}$$
for some scalars $c_{\beta,i}$.

Let $S$ be an intermediate subgroup with $H\subseteq S \subsetneq G$ and consider the decomposition of $\Lie(G)$ given by $\Lie(G)=\Lie(S)\oplus V_{S}$ where $V_S$ is $\Ad(\Lie(S))$-invariant, but not necessarily irreducible. If we decompose $V_S$ into $\Ad(\Lie(S))$-invariant subspaces, then each subspace will be non-trivial since $C_G(H)$ is finite. This will be an important fact that we use throughout the paper when working with the adjoint representation.

We end the section by introducing two results from Einsiedler--Margulis--Venkatesh \cite{MR2507639} on intermediate subgroups $H \subset G$. 
\begin{lem}[Lemma 3.4.1, \cite{MR2507639}]\label{finitely_many_intermediate_subgroups}
Suppose $H\subseteq G$ are semisimple Lie groups without compact factors such that $C_G(H)$ is finite. Then there are only finitely many intermediate subgroups $H\subseteq S \subsetneq G$. Each such $S$ is semisimple and without compact factors.
\end{lem}
\begin{lem}[Appendix A, \cite{MR2507639}]\label{classifyH} If $G$ is a semisimple Lie group without compact factors, then there exists a finite collection of semisimple subgroups $\mathcal H$ such that the following holds: for any semisimple Lie subgroup $H$ $G$ with no compact factors and $C_G(H)$ finite, there exists $H' \in \mathcal H$ and $g \in G$ such that $H = g H'g^{-1}$.

\end{lem}


\section{Linear Algebra Lemma}
In this section we state some key technical lemmas related to the action of horospherical subgroups and diagonal subgroups from \cite{katz/online} and  \cite{MR1403756} and prove extensions of these results. The main result of this section  applies to representations that are not necessarily irreducible.

\begin{lem}[Linear algebra lemma]\label{Lem:General_LA_lemma}
 Suppose $\rho:G\to GL(V)$ is a faithful finite dimensional  representation of a semisimple Lie group $G$. Suppose $V$ decomposes into non-trivial and irreducible subspaces $V=\oplus_iV_i$. There exists an absolute constant $0<\delta_0 = \delta_0(\dim(G))\ll 1$ such that for all $0<\delta<\delta_0$ and $0<c<1$, there exists $t_{\delta,c}= t_{\delta,c}(G, H)>0$  with
$$\frac{1}{m_U(B_2^U)}\int_{B_2 ^{U}}\frac{1}{\|\rho(a_{ t}u)v\|^\delta}\,dm_U(u)< \frac{c}{\|v\|^\delta}$$
for every $v\in V$, $t\ge t_{\delta,c}$.
\end{lem}

We now state some lemmas from Katz \cite{katz/online}. In brief, these lemmas show that the action of the diagonal and horspherical subgroups on a vector space expand the norm. While we will follow the exposition of Katz \cite{katz/online}, we would like to draw the readers attention to Shah \cite{MR1403756}, specifically Section 5.

The following lemmas are essentially Lemma 3.1 and Lemma 3.2 of \cite{katz/online}.

\begin{lem}[Lemma 3.1, \cite{katz/online}]
 Let $u = \exp(\underline{u}) \in B_r ^{U}.$ There exists polynomials $f_{\beta,j}:B_r ^{\Lie(U)}\to \mathbb R$ with
$$\rho(u)v = \sum_{\beta\in \Phi} \sum_{j=1} ^{\text{\rm{dim}} V_\beta}f_{\beta,j}(\underline{u})v_{\beta,j}.$$
\end{lem}

\begin{proof}
The proof of Lemma 3.1 in \cite{katz/online} works for any $u\in B_r ^U$ where $r>0$.
\end{proof}

\begin{lem}[Anchor Lemma, Lemma 3.2, \cite{katz/online}]\label{Lem:Katz_Non-trivial_projection} Let $\rho:G\to GL(V)$ be a finite dimensional irreducible representation of a semisimple Lie group $G$. Then for any $r>0$ and non-zero $v\in V$, there is a positive root $\beta\in \Phi^+$ and $1\le j\le \text{dim }V_\beta$ such that $$\sup_{u\in B_r ^{U}}|f_{\beta,j}(u)|>0.$$
\end{lem}

\begin{proof}
The proof of Lemma 3.2 in \cite{katz/online} works for any open ball $B_r ^U$ with $r>0$.
\end{proof}

By the Anchor lemma, the projection of the action of $U$ in the expanding direction is nonzero. Thus, the norm under the action of $a_{t}$ grows. By noting that Lemmas 3.1 and 3.2 of Katz \cite{katz/online} hold for any open ball $B_r^U$, we have the following minor generalization of Lemma 2.3 of \cite{katz/online}.

\begin{lem}\label{Lem:Katz_irreducible_LA_lemma}
 Suppose $\rho:G\to GL(V)$ is an irreducible finite dimensional  representation of a semisimple Lie group $G$. There exists $0<\delta_0 = \delta_0(\text{\rm{dim}\it}(G))\ll 1$ such that for all $0<\delta<\delta_0$ and $0<c<1$, there exists $t_c>0$ with
$$\frac{1}{m_U(B_2^U)}\int_{B_2^U}\frac{1}{\|\rho(a_{ t} u)v\|^\delta}\,dm_U(u)< \frac{c}{\|v\|^\delta}$$
for every $v\in V$ and $t\ge t_c$.
\end{lem}

We conclude with a proof of the lemma stated at the beginning of the section.

\begin{proof}[Proof of Lemma \ref{Lem:General_LA_lemma}]

This essentially follows from the irreducible version of \cite{katz/online}. We equip $V=\oplus_iV_i$ with the max norm. That is, for $v=(v_i)$, $\|v\|=\max_i\|v_i\|$.  We also note that the inequality we aim to prove is independent of the choice of norm.
 
 Let $0<c<1$ and $0<\delta<\delta_0$. We will choose $\delta_0$ in the course of the proof.

Given $v=(v_i)$, let $i_0$ be the index with $\|v\|=\|v_{i_0}\|$. Then,
$$\|\rho(a_{t}u)v\|=\max_i \|\rho(a_{t}  u)v_i\|\ge \|\rho(a_{t}u)v_{i_0}\|.$$
By the irreducible case (Lemma \ref{Lem:Katz_irreducible_LA_lemma}), we have the existence of $\delta_i\in(0,1)$ such that for every $\delta\in(0,\delta_i)$ contraction occurs for every $v_i\in V_i$ and $t$ sufficiently large. To finish the proof, take $\delta_0 :=\min_i \delta_i$ and we have
\begin{align*}
\frac{1}{m_U(B_2 ^U)}\int_{B_2 ^{U}}\frac{1}{\|\rho(a_{ t} u)v\|^\delta}\,dm_U(u)&\le \frac{1}{m_U(B_2 ^U)} \int_{B_2 ^{U}}\frac{1}{\|\rho(a_{ t} u)v_{i_0}\|^\delta}\,dm_U(u)\\
&< \frac{c}{\|v_{i_0}\|^\delta}=\frac{c}{\|v\|^\delta}
\end{align*}
for every $v\in V$ and $t$ sufficiently large.
\end{proof}

Now we apply the above Linear algebra lemma to a specific representation which will be used to control the height function. First we give a definition. See also the end of section 2 of Eskin--Margulis \cite{MR2087794}. 

\begin{defn}[Maximal Parabolic Subgroups]\label{Def:Height_rep}
 When a lattice $\Gamma$ is non-uniform, we define a finite collection $\Delta$ of maximal parabolic subgroups of $G$ as follows. A parabolic subgroup $P$ of $G$ is called $\Gamma$-rational if $\,\Gamma \cap R_u(P)$ is a lattice in $R_u(P)$, where $R_u(P)$ is the unipotent radical of $P$. If $G$ is of real rank 1, then we let $\Delta = \{P_0\}$ where $P_0$ is a $\Gamma$-rational minimal parabolic subgroup of $G$. The existence of $P_0$ follows from Garland--Raghunathan \cite{MR267041}. If the real rank of $G$ is not less than $2$, then by the Margulis Arithmiticity theorem, $\Gamma$ is arithmetic. Hence we let $\Delta=\{P_1,P_2,...,P_r\}$ where $P_k$ are standard parabolic subgroups of $G$, with respect to its maximal $\mathbb Q$-split torus $A_0$. For every $1 \le k \le r$, there exists a finite-dimensional irreducible representation $\rho_k:G\to GL(W_k)$ and vectors $w_k\in W_k$ such that the stabilizer of $\mathbb R w_k$ is $P_k$.
\end{defn}
 
\begin{remark}[Upper bound on the dimension of $W_k$]\label{Wk} For later computational purposes (see Lemma \ref{Lem:Log_continuity_height} and Lemma \ref{kappa}), we take $W_k$ to have dimension no greater than $\dim(G)^2$. We can do so by choosing
$$W_k := \wedge^{\dim(R_u(P_k))}\Lie(R_u(P_k)) \subseteq \wedge^{\dim(R_u(P_k))}\Lie(G) $$
and $w_k$ to be a normalized diagonal element of $W_k$.

\end{remark}

The following result is an analogue of Condition A of Eskin--Margulis \cite{MR2087794}. We can deduce it by applying Lemma \ref{Lem:General_LA_lemma} to the representation above.

\begin{lem}[Linear algebra lemma for Height functions]\label{Lem:LA_lemma} If $\Gamma$ is a non-uniform lattice of $G$, then there exists $0<\delta_1 = \delta_1(\dim(G))\ll 1$ such that for all $0<\delta<\delta_1$ and $0<c<1$, there exists $t_{\delta,c} = t_{\delta,c}(G, H)>0$ such that for every $v\in Gw_k$ and $t\ge t_{\delta,c}$,
$$\frac{1}{m_U(B_2^U)}\int_{B_2 ^{\Lie(U)}}\frac{1}{\|\rho_k(a_{ t}u)v\|^\delta}\,dm_U(u)< \frac{c}{\|v\|^\delta}.$$
\end{lem}


\section{Abstract Margulis inequality}

In this section we prove an abstract result that yields exponential decay for Margulis functions.

\begin{thm}\label{Thm:Abstract_exp_decay}
Suppose $F:X\to(0,\infty)$ satisfies the following properties
\begin{itemize}
    \item (Log Continuity) For any compact subset $K\subset G$, there exists $\sigma=\sigma_F(K)>1$ such that for all $g\in K$ and $x\in X$,
    $$\sigma^{-1} F(x)\le F(gx)\le \sigma F(x)$$
    \item (Margulis Inequality for F) There exists constants $0< c <1$, $t=t_c \gg 1$, and $b>0$ such that for any $x\in X$,
    $$(A_{2,t}F)(x) := \frac{1}{m_U(B_2^U)}\int_{B_2 ^{U}}F(a_tux)\,dm_U(u) < c F(x) + b.$$
\end{itemize}
Then, there exists absolute constants $C>0$ and $B>0$ such that for any $t \ge t_c$,
$$(A_{t} F)\le C \cdot c^{t/t_c} F+B.$$
\end{thm}

\begin{proof}
\textbf{Step 1:}
We find an upper bound on $A_{nt} F$.

Recall, $$(A_{2,t}f)(x) := \frac{1}{m_U(B_2 ^U)}\int_{B_2 ^U}F(a_{t}u x)\,dm_U(u).$$

By iterating our operator, we have $$(A_{2,t} ^n F)(x)<c^nF(x)+B'$$
where $B'=O(b)=b\sum_{j=0} ^n c^j$.
On the other hand, 
$$(A_{2,t} ^n F)(x) = \frac{1}{m_U ^n(B_2 ^U)}\int_{(B_2 ^U)^{n-1}}\left(\frac{1}{m_U(B_2 ^U)}\int_{B_2 ^U}F(a_{nt}\phi_n(\vec{u})u_1x)\,dm_U(u_1)\right)\,d(m_U)^{n-1}(\vec{u})$$
where $\phi:(B_2 ^U)^{n-1}\to U$ is given by
$$\vec{u}=(u_n,\ldots, u_2)  \mapsto  \phi_n(\vec{u}) = (a_{-(n-1)t} u_n a_{(n-1)t}) \cdots (a_{-2t}u_3a_{2t})(a_{-t}u_2a_t).$$
Hence, there must exist some $\vec{u}\in (B_2 ^U)^{n-1}$ so that 
$$\frac{1}{m_U(B_2 ^U)}\int_{B_2 ^U}F(a_{nt}\phi_n(\vec{u})u_1x)\,dm_U(u_1)<c^nF(x)+B'.$$

Now note that since $t$ is large, $\phi((B_2 ^U)^{n-1}) \subseteq B_1^U$ and so  $\phi(\vec{u})^{-1}B_1 ^U\subseteq B_2 ^U$. Hence,
$$\frac{1}{m_U(B_2 ^U)}\int_{\phi(\vec{u})^{-1}B_1 ^U}F(a_{nt}\phi_n(\vec{u})u_1x)\,dm_U(u_1)\le \frac{1}{m_U(B_2 ^U)} \int_{B_2 ^U}F(a_{nt}\phi_n(\vec{u})u_1x)\,dm_U(u_1).$$
Make the substitution $v=\phi(\vec{u})^{-1}u$ and note that $m_U$ is translation invariant to obtain
$$\frac{1}{m_U(B_2 ^U)}\int_{\phi(\vec{u})^{-1}B_1 ^U}F(a_{nt}\phi_n(\vec{u})u_1x)\,dm_U(u_1) = \frac{1}{m_U(B_2 ^U)} \int_{B_1 ^U}F(a_{nt}vx)\,dm_U(v).$$

Putting everything together, we obtain 
$$(A_{nt} F)(x) = \int_{B_1 ^U}F(a_{nt}vx)\,dm_U(v)< m_U(B_2 ^U)c^nF(x)+B$$
where $B=B'm_U(B_2 ^U)$.

\textbf{Step 2:} Now we use the previous step to show for arbitrarily large $t\ge t_c$ we have a bound on $A_tF$.

Let $K_1 = \{a_t:0 \le t \le t_c\}$ be a fixed compact set and let $\sigma_1 = \sigma_F(K_1)$ be the constant from the log continuity property of $F$. Suppose that $t \ge t_c$, and let $\lfloor t/t_c\rfloor=n$. Since $0 \le t-nt_c < t_c$, by log-continuity of $F$,
$$\int_{B_1 ^U} F(a_{t}ux)\,dm_U(u)=\int_{B_1 ^U} F(a_{t-nt_c}\cdot a_{nt_c}ux)\,dm_U(u)\le \sigma_1\int_{B_1 ^U} F(a_{nt_c}ux)\,dm_U(u).$$

By \textbf{Step 1},
$$\int_{B_1 ^U} F(a_{nt_c}ux)\,dm_U(u) = (A_{nt}F)(x) \le m_U(B_2 ^U)c^n F(x) + B$$
and thus,
\begin{align*}  
\int_{B_1 ^U} F(a_{t}ux)\,dm_U(u) &\le \sigma_1 \left(c^n m_U(B_2 ^U)F(x) + B\right) =  \sigma_1m_U(B_2 ^U) c^{\lfloor t/t_c\rfloor}F(x) + \sigma_1 B \\
& \le \sigma_1m_U(B_2 ^U) c^{t/t_c -1} F(x) + \sigma_1 B.
\end{align*}
Letting $C = \sigma_1m_U(B_2 ^U)c^{-1}$ and relabeling $\sigma_1 B$ as $B$ finishes the proof.
\end{proof}


\section{Height functions and Margulis inequality}

Throughout this section, $\Gamma$ will be taken to be a non-uniform lattice of $G$. Thus, the space $X=G/\Gamma$ has a cuspidal part and we construct a height function $h$ on $X$ that measures how high a point $x\in X$ is in the cusp. We prove that $h$ satisfies a Margulis inequality. The height function in use is essentially the same as in Eskin--Margulis \cite{MR2087794}. However, instead of taking the average over some random walk on $X$, we average over the expanding translates of $B_1^U$, the unit ball in the horospherical subgroup.

\begin{thm}\label{Thm:Margulis_height} For any $0<\delta<\delta_1$ ($\delta_1$ as in Lemma \ref{Lem:LA_lemma}), there exists a height function $h = h_\delta : X \rightarrow (0, \infty)$ such that the following holds: for any $0<c<1$, there exists $t_c >0$ such that for any $t \ge t_c$, there exists absolute constant $B_t >0$ such that
$$A_{2,t}h < ch + B_t .$$
\end{thm}

Note that Theorem \ref{Thm:Margulis_height} directly implies the (Margulis inequality) hypothesis in Theorem \ref{Thm:Abstract_exp_decay} for $h$. In Lemma \ref{Lem:Log_continuity_height} we shall see that $h$ is log-conitnuous also, and thus we get the following exponential decay property for $h$. 

\begin{cor}\label{Corollary:Exponential_decay_height} For any $0<\delta<\delta_1$, height function $h = h_\delta$ (the same height function as in Theorem \ref{Thm:Margulis_height}) satisfies the following: there exists $t_h > 0$, $C_h>0$ and $B_h>0$ such that for all $t \ge t_h$,
$$A_{t}h \le \frac{C_h}{2^{t/t_h}} \cdot h +B_h.$$
\end{cor}

\begin{proof} 
The result directly follows from Theorem \ref{Thm:Abstract_exp_decay}, Theorem \ref{Thm:Margulis_height}, and Lemma \ref{Lem:Log_continuity_height}. We note that $t_h$ is equal to $t_{1/2}$, defined as in Theorem \ref{Thm:Margulis_height}.

\end{proof}

\textbf{Construction of the height function $h$.} Let $P_0$ denote a minimal $\Gamma$-rational parabolic subgroup of $G$. Then we have the Langlands decomposition $P_0=M_0A_0N_0$ where $M_0$ is semisimple, $A_0$ is abelian, and $N_0$ is the unipotent radical of $P_0$. If $G$ has real rank greater than 2, then we let $A_0$ to be the fixed maximal $\mathbb Q$-split torus of $G$ in Definition \ref{Def:Height_rep}. Let $\mathfrak{a}$ denote the Lie algebra of $A_0$. We shall identify $\mathfrak{a}$ with its dual via the Killing form. Let $\alpha_1,\alpha_2,\ldots,\alpha_r$ denote the roots which we view as elements of the dual of $\mathfrak{a}$. A \emph{Siegel set} is a set $\mathfrak{S}=K\mathcal{M}\mathcal{A}\mathcal{N}$ where $K$ is the maximal compact subgroup of $G$, $\mathcal{M} \subseteq M_0$ and $\mathcal{N}\subseteq N_0$ are compact, and $\mathcal{A}=\P\{a\in A_0:\alpha_k(\log a)<C \text{ for all }1\le k\le r\}$ for some positive constant $C$.

Note that for appropriate choices of $\mathcal{M}$, $\mathcal{N}$, and $C$, there exists a finite set $J\subseteq G$ such that for every $g\in G$, the intersection $\mathfrak{S}\cap g\Gamma J$ is not empty. See Dani--Margulis \cite{MR1237827} for details.

For $1 \le k \le r$, define $d_k(g):=\|\rho_k(g)w_k\|$ where $\rho_k:G\to GL(W_k)$ and $w_{k}\in W_k$ are defined as in Definition \ref{Def:Height_rep}. By structure theory, there exists absolute constants $C_0$ and $c_1, c_2, ..., c_r$ such that for each $1 \le k \le r$, 

$$ d_k(g) = d_k(a)  \text{\,\,\,and\,\,\,}  |\log(d_k(a))-c_k\omega_k(\log a)|<C_0$$ 

for all $g\in G$ where $g=kman$ is the Langlands decomposition of $g$ with respect to $P_k$ and $\omega_k$ is the co-root corresponding to $\alpha_k$; i.e. $\omega_k(\alpha_k)=1$ and $\omega_k(\alpha_j)=0$ for all $j \neq k$. Let 
$$\beta_k(g)=\max_{\gamma\in \Gamma}\frac{1}{d_k(g\gamma)^{1/c_k}}.$$ 

Also, for $x = g\Gamma \in X$, we shall define $\beta_k(x) := \beta_k(g)$.

\begin{remark}\label{Rmk:str_theory}
There exists an absolute constant $C=C(G/\Gamma)>1$ such that for any $g \in G$ and $g_1 \in \mathfrak{S} \cap g\Gamma J$, 
$$C^{-1}\beta_k(g_1) < \beta_k(g) < C\beta_k(g_1).$$ 
\end{remark}

Lastly, we choose a sequence of positive real numbers $\{q_k \}_{k=1}^{r}$ so that $\sum_kq_k\omega_k$ belongs to the postive Weyl chamber of $\mathfrak{a}$ and let $$h_k(g)=\beta_k(g)^{1/q_k}.$$ 
For later use (see Proposition \ref{Prop:Upper_bound_height}), we shall take $\{q_k \}_{k=1}^{r}$ to be normalized so that $$\min_{1\le k\le r}\{c_kq_k\} =1.$$
Our height function $h$ will be defined to be 
$$h := C_* \sum_{k=1}^{r} h_{k}^{\delta_*}$$ 
for some $\delta_*$ and $C_*$. The condition $\delta_* \ll 1$ will be later verified in the proof of Proposition \ref{Prop:Upper_bound_height} so that $h$ satisfies the Margulis inequality, and the constant $C_* \gg 1$  will be later chosen in Remark \ref{Rmk:lowerbound_h} so that $h$ is bounded away from 1.

\begin{lem}[Log continuity of height function] \label{Lem:Log_continuity_height} For any compact subset $K \subset G$, there exists $\sigma_h$ = $\sigma_h(K) \ge 1$ such that for all $g \in K$ and $x \in X$, 
$$\sigma_h^{-1} \cdot h(x) \le h(gx) \le \sigma_h \cdot h(x).$$
Moreover, $\sigma_h$ can be chosen to be a constant only depending on a compact set $K$ and $\dim(G)$, each of which will be independent of lattice $\Gamma$.
\end{lem}
\begin{proof}
For each $1\le k \le r$, the map $d_k = \|\rho_k( \cdot )w_k\|$ is log continuous. Thus, there exists $\sigma_k = \sigma_k(K)$ such that for all $g' \in K$ and $g \in G$, 
$$\sigma_k^{-1} \cdot d_k(g) \le d_k(g'g) \le \sigma_k \cdot d_k(g).$$

Note that $\sigma_k$ only depends on compact set $K$ and $\dim(W_k).$ By Remark \ref{Wk}, $ \dim(W_k).\le \dim(G)^2)$ so that $\sigma:= \max_{1\le k \le r} \{\sigma_k\}$ is a constant only depending on $K$ and $\dim(G)$. 

Now for each $1\le k \le r$, let $\sigma_k' := \sigma^{1/c_k}$, where $c_k$ are the constants used to define $\beta_k$. Then, for all $g' \in K$ and $g \in G$,
$$(\sigma_k')^{-1} \cdot \beta_k(g) \le \beta_k(g'g) \le \sigma_k' \cdot \beta_k(g).$$

If the maximum in $\beta_k(g'g)$ is achieved by the same $\gamma \in \Gamma$ as in $\beta_k(g)$, then the result directly follows from log continuity of $d_k$ and definition of $\sigma_k'$. Suppose that the maximum is achieved by different choice of $\gamma$;
$$\beta_k(g) = \frac{1}{d_k(g\gamma_0)^{1/c_k}} \ \text{ and }\ \beta_k(g'g) = \frac{1}{d_k(g'g\gamma_1)^{1/c_k}}$$
for some $\gamma_0 \neq \gamma_1 \in \Gamma$. Then, 
$$\beta_k(g'g) > \frac{1}{d_k(g'g\gamma_0)^{1/c_k}} = d_k(g'g\gamma_0)^{-1/c_k} \ge (\sigma \cdot d_k(g\gamma_0))^{-1/c_k} = (\sigma_k')^{-1} \cdot \beta_k(g)$$
and
$$\beta_k(g'g) = d_k(g'g\gamma_1)^{-1/c_k} \le (\sigma^{-1} \cdot d_k(g\gamma_1))^{-1/c_k} <  \sigma_k' \cdot d_k(g\gamma_0))^{-1/c_k} = \sigma_k' \cdot \beta_k(g).$$

Lastly, by definition $$h := C_* \sum_{k=1}^{r} h_{k}^{\delta_*} = C_* \sum_{k=1}^{r} \beta_{k}^{\delta_*/q_k},$$ 
so it follows that the height function $h$ has log continuity  with constant 
$$\sigma_h := \max_{1\le k\le r}\{(\sigma_k')^{\delta_*/q_k}\} = \sigma^{\max_{1\le k\le r}\{\delta_*/c_kq_k\}} = \sigma^{\delta_*/{\min_{1\le k\le r}\{c_kq_k\}}} = \sigma^{\delta_*}.$$
(Here we are using the fact that the $\{q_k \}_{k=1}^{r}$ are normalized to satisfy $\min_{1\le k\le r}\{c_kq_k\} =1$.) Since $\delta_* = \frac{1}{2}\delta_1$ is a constant only dependent on $\dim(G)$ (see Lemma \ref{Lem:LA_lemma}), $\sigma_h$ is only dependent on compact set $K$ and $\dim(G)$.

\end{proof}


The following lemma is a direct result of equation (33) from Eskin--Margulis \cite{MR2087794}.

\begin{lem}\label{Lem:bound_h_k}
For any constant $C>1$, there exists an absolute constant $D_C = D_C(G/\Gamma)>1$ such that the following holds:
if for some $1 \le k \le r$ and $g \in G$ there exists $g_1 \in \mathfrak{S} \cap g\Gamma J$ and $g_2 (\neq g_1) \in g\Gamma J$ such that $d_k(g_2) < Cd_k(g_1)$, then

$$h_k(g) \le D_C \prod_{j\ne k}h_j ^{ \lambda_{j,k}}(g)$$
where $\lambda_{j,k} = \frac{q_j| < \alpha_j, \alpha_k > |}{q_k<\alpha_k,\alpha_k>}$.
\end{lem}

In view of Remark \ref{Rmk:str_theory}, we can rewrite the above Lemma \ref{Lem:bound_h_k} as follows. 

\begin{lem}[Upper bound for $h_k$]\label{Lem:bound_h_k_refined} For any constant $C>1$, there exists an absolute constant $D_C' = D_C'(G/\Gamma)>1$ such that the following holds. If $\beta_k(g) = \frac{1}{d_k(g\gamma_0)^{1/c_k}}$ for some $\gamma_0 \in \Gamma$ and there exists $\gamma_1 (\neq \gamma_0) \in \Gamma$ such that $d_k(g\gamma_1)<Cd_k(g\gamma_0)$, then 

$$h_k(g) \le D_C'\prod_{j\ne k}h_j ^{ \lambda_{j,k}}(g).$$
\end{lem}

Now, we replace Condition A of Eskin--Margulis \cite{MR2087794} with Lemma \ref{Lem:LA_lemma} and prove an Margulis inequality for our averaging operator $A_{t}$.

\begin{prop}[Upper bound for $A_{t}h_k^\delta$]\label{Prop:Upper_bound_height}
 For any  $0<\delta<\delta_1$ and $0<c<1$, for all $t \ge t_{\delta,c}$ (where $t_{\delta, c}$ is as in Lemma \ref{Lem:LA_lemma}), there exists an absolute constant $D_t >0$ (depending only on $t$) such that  for any $1\le k\le r$,
$$(A_{2,t}h_k^\delta)(g) := \frac{1}{m_U(B_2^U)}\int_{B_2 ^{U}}h_k^{\delta}(a_t u g)\,dm_U(u) \le ch_k^\delta(g) + D_t \prod_{j\ne k}h_j ^{\delta \lambda_{j,k}}(g)$$
for any $g \in G$.
\end{prop}

\begin{proof}
If for every $a_tug$ (varying $u$ over $B_2^{U}$) the maximum in
$\beta_k$ is achieved by the exact same $\gamma \in \Gamma$ as in $\beta_k(g)$, then we get $(A_{2,t}h_k^\delta)(g) \le ch_k^\delta(g)$ directly from Lemma \ref{Lem:LA_lemma}, since 
$$h_k(g) = \frac{1}{d_k(g)^{(\delta/c_k q_k)}} = \frac{1}{\|\rho_k(g)w_k\|^{(\delta/c_k q_k)}} $$

and $\delta/c_k q_k < \delta_1/c_k q_k \le \delta_1$. (Here we are using the fact that the $\{q_k \}_{k=1}^{r}$ are normalized to satisfy $\min_{1\le k\le r}\{c_kq_k\} =1$.)

Suppose that the maximum in $\beta_k$ is achieved by a different $\gamma$ for some $u \in B_2 ^{U}$;
$$\beta_k(g) = \frac{1}{d_k(g\gamma_0)^{1/c_k}} \ \text{ and }\ \beta_k(g'g) = \frac{1}{d_k(g'g\gamma_1)^{1/c_k}}$$

for some $g' \in a_t B_2 ^{U}$ and $\gamma_0 \neq \gamma_1 \in \Gamma$. By definition of $\beta_k$, compactness of $a_t B_2 ^{U}$, and log continuity of $d_k$, we have
$$d_k(g\gamma_1) < Cd_k(g'g\gamma_1) < Cd_k(g'g\gamma_0) < C^2d_k(g\gamma_0),$$

where $C = C_{t,k}=\sigma_{d_k}(a_t B_2 ^{U})$. By Lemma \ref{Lem:bound_h_k_refined} and by log continuity of $h_k$,
\begin{align*}
\frac{1}{m_U(B_2^U)}\int_{B_2 ^{U}}h_k^{\delta}(a_t ug)\,dm_U(u)
&\le \frac{1}{m_U(B_2^U)}\int_{B_2 ^{U}} C' h_k^{\delta}(g)\,dm_U(u) = C' h_k^{\delta}(g)\\
&\le C'D_{C^2}'\prod_{j\ne k}h_j ^{\delta \lambda_{j,k}}(g)
\end{align*}

where $C'= C_{t,k}'= \sigma_{h_k}(a_t B_1^{U})$, and $D_{C^2}'$ is defined as in Lemma \ref{Lem:bound_h_k_refined}, for our choice of $C=C_{t,k}$. We note that $C=C_{t,k}$, $C'=C_{t,k}'$, and $D_{C^2}' = D_{C_{t,k}^2}'$ are all absolute constants, only depending on $t$ and $k$. Thus, by taking $D_t:=\max_{1\le k\le r} \{C_{t,k}'D_{C_{t,k}^2}'\}$, we get the desired result. 
\end{proof}

\begin{thm}\label{Theorem:Margulis_Inequality_height} 
For any triple of $0<\delta<\delta_1$, $0<c<1$, and $t \ge t_{\delta,c/2}$ ($t_{\delta, c/2}$ as in Lemma \ref{Lem:LA_lemma}), there exists $0<\varepsilon = \varepsilon_{t, c} \ll 1$ such that the height function $ h = h_{t, c} := \sum_{k=1}^{r}(\varepsilon h_k)^\delta$
satisfies$$A_{2,t} h \le\, ch + 1 .$$
\end{thm}
\begin{proof}
By Proposition \ref{Prop:Upper_bound_height}, we have
$$A_{2,t}h_k^\delta \le \frac{c}{2} \cdot h_k^\delta + D_t \cdot \prod_{j\ne k}h_j ^{\delta \lambda_{j,k}}.$$
Taking the sum over all $1\le k\le r$ and multiplying $\varepsilon^\delta$ to both sides of the above equation yields,
$$A_{2,t} \left( \sum_{k=1}^r (\varepsilon h_k)^\delta\right) \le \frac{c}{2} \cdot \sum_{k=1}^r (\varepsilon h_k)^\delta + D_t \cdot \sum_{k=1}^r  \varepsilon_k\prod_{j\ne k}(\varepsilon h_j)^{\delta \lambda_{j,k}}$$
where $\varepsilon_k = \varepsilon^{\delta \cdot (1 - \sum_{j\neq k}\lambda_{j,k})}$ for each $1\le k\le r$. Since $\sum_k q_k\omega_k$ belongs to the positive Weyl chamber, we have that $\sum_{j\ne k}\lambda_{j,k} < 1$ for each $1\le k\le r$. Hence, by Jensen's inequality
\begin{align*}
\prod_{j\ne k}(\varepsilon h_j)^{\delta \lambda_{j,k}} & =\exp \left( \sum_{j\ne k} \lambda_{j,k} \cdot \left( \log   (\varepsilon h_j)^{\delta} \right) + \left(1- \sum_{j\ne k} \lambda_{j,k} \right) \cdot 0 \right) \\
&\le  \sum_{j\ne k} \lambda_{j,k} \cdot \exp \left( \log   (\varepsilon h_j)^{\delta} \right) + \left(1- \sum_{j\ne k} \lambda_{j,k} \right) \cdot \exp (0) \\
&= \sum_{j\ne k} \lambda_{j,k} (\varepsilon h_j)^{\delta} + (1 - \sum_{j\neq k}\lambda_{j,k}) \cdot 1 \le \sum_{j\ne k} (\varepsilon h_j)^{\delta}+1
\end{align*}

and thus,
\begin{align*}
A_{2,t} \left(\sum_{k=1}^r (\varepsilon h_k)^\delta\right) &\le \frac{c}{2} \cdot \sum_{k=1}^r (\varepsilon h_k)^\delta + \left(D_t  \sum_{k=1}^r  \varepsilon_k \right) \cdot \left(\sum_{k=1}^r (\varepsilon h_k)^\delta +1\right)\\
&= \left(\frac{c}{2} + D_t \sum_{k=1}^r  \varepsilon_k \right) \cdot \sum_{k=1}^r (\varepsilon h_k)^\delta+\left(D_t  \sum_{k=1}^r  \varepsilon_k \right)
\end{align*}

Since $\varepsilon_k = \varepsilon^{\delta \cdot (1 - \sum_{j\neq k}\lambda_{j,k})}$ and $\sum_{j\ne k}\lambda_{j,k} < 1$ for each $1\le k\le r$, we can choose $\varepsilon = \varepsilon_{\delta,c} \ll 1$ small enough so that $\left(D_t  \sum_{k=1}^{r}  \varepsilon_k \right) < \frac{c}{2} < 1$.
\end{proof}

\begin{proof}[Proof of Theorem \ref{Thm:Margulis_height}] For fixed $0<\delta<\delta_1$, the class of height functions $\{h_{t,c}\}$  with $0<c<1$ and $t \ge t_{\delta, c/2}$ defined as in Theorem \ref{Theorem:Margulis_Inequality_height} are linear. That is, if we let $h = h_{\delta} = \sum_{k=1}^{r} h_k^{\delta}$, then
$$h_{t,c} = \sum_{k=1}^{r} (\varepsilon_{t,c} h_k)^{\delta} = \varepsilon_{t,c}^{\delta} \cdot \sum_{k=1}^{r} h_k^{\delta} = \varepsilon_{t,c}^{\delta} \cdot h.$$
By Theorem \ref{Theorem:Margulis_Inequality_height}, for any pair of $0<c<1$ and $t \ge t_{\delta, c/2}$, we have 
$$(A_{2,t} h_{t,c})(x) \,\le\, ch_{t,c}(x) + 1.$$

Multiply $B_t := 1/\varepsilon_{t,c}^{\delta}$ to both sides of the inequality and we get
$$(A_{2,t} h)(x) \,\le\, ch(x) + B_t.$$
Lastly, take $t_c$ = $t_{\delta,c/2}$ and we are done. 
\end{proof}

\begin{remark}[Lower bound for $h$]\label{Rmk:lowerbound_h} For  computational reasons that will become apparent later, we want our $h$ to be large and bounded away from 1. Since $\{d_k(g):g \in \mathfrak{S}\}$ ($\mathfrak{S}$ is the Siegel set) is bounded away from zero, by Remark \ref{Rmk:str_theory} we have that for each $1\le k\le r$, $h_k$ is bounded away from zero and  therefore, $h_{\delta} = \sum_{k=1}^{r} h_k^{\delta}$ is bounded away from zero. Thus, by multiplying some large $C \gg 1$, we can make our height function $h:=Ch_{\delta}$ to be no less than 2. Note that constant multiple does not effect Margulis inequalities (see proof of Theorem \ref{Thm:Margulis_height}) except that it only makes the additive constant $B_t$ bigger; our newly defined $h=Ch_{\delta}$ also satisfies Theorem \ref{Thm:Margulis_height} and Corollary \ref{Corollary:Exponential_decay_height}. 

For the remainder of the paper, if $\Gamma$ is a non-uniform lattice, then we fix $\delta_* := \frac{1}{2} \delta_1$ and let our height function to be $h=C_*h_{\delta_*} \ge 2$. However, note that the exact value of $\delta_*$ is not important and all arguments in the later sections apply for any choice of $0<\delta_*<\delta_1$.
\end{remark}

\section{Return lemma and Number of Nearby Sheets}\label{Return lemma and Number of Nearby Sheets}

Let $H \subseteq S \subsetneq G$ be an intermediate orbit. For a closed $S$-orbit $Y=Sy$ and point $x \in X$, we shall define a window set $I_Y(x)$ which collects all the sheets of $Y$ that are nearby $x$. Roughly, the idea is to collect sheets of $Y$ within the injective ball $B_{\inj(x)}^G(x)$, but the exact size of our windows will be much smaller, and will be given in terms of the height function $h$. A formal definition of $I_Y(x)$ will be given in Definition \ref{Def:Window_I_Y}.

The aim of this section is to show that $\#I_Y(X)$, the number of nearby sheets, is bounded in terms of volume of $Y$.
\begin{prop}[Number of nearby sheets]\label{Prop:Number_of_Nearby_Sheets}
 There exists a global constant $\constE\label{C:sheets}>0$ such that for any intermediate subgroup $H\subseteq S\subsetneq G$, closed $S$-orbit $Y=Sy$, and $x\in X$, we have
$$\#I_Y(x)<C_1 \vol(Y).$$
\end{prop}

\subsection{Height function and Radius of Injectivity}
First we compare $h(x)$, the value of the our height function at $x \in X$, with $\inj(x)$, the injectivity radius at point $x$. We note that the following Proposition is as an analog of Lemma 6.3. of Benoist-Quint \cite{MR2831114}.

\begin{prop}\label{Proposition:comparing_height_and_inj_radius}
If $\Gamma$ is a non-uniform lattice of $G$, then there is absolute constants $\constE\label{C:comparing_height_and_inj_radius}>0$ and $m>0$ such that for  all $x\in X$,
$$\inj(x)^{-1}\le \ref{C:comparing_height_and_inj_radius} h(x)^{m}.$$
\end{prop}

\begin{proof}
Let $\dist_G$ denote the left invariant Riemannian metric on $G$. Suppose that for some $g_1 \neq g_2 \in B_{\epsilon}^G(e)$ and $x=g\Gamma \in X$, $g_1x=g_2x \in X$. Then, for any $\gamma \in \Gamma$, $(\gamma g)^{-1}(g_1^{-1}g_2)(g\gamma)$ is in $\Gamma$.
Moreover, for each point in $G$, there is a neighborhood on which the metric $\dist_G$ is Lipschitz equivalent to the metric derived from matrix norm. Thus, 
$$2\epsilon \ge \dist_G(e, g_1^{-1}g_2) \ge \dist_G(e, (\gamma g)^{-1}(g_1^{-1}g_2)g\gamma) ||\Ad(g\gamma)^{-1}||^{-1}.$$

Since $\Gamma$ is a lattice, $\inf_{e\neq\gamma \in \Gamma}\ \dist_G(e,\gamma) >0$ and thus, $\inj(x) \gg \min_{\gamma \in \Gamma}{||\Ad(g\gamma)^{-1}||}^{-1}.$ 

Since $\mathfrak{S}\cap g\Gamma J$ is non-empty, we choose $g' \in \mathfrak{S}\cap g\Gamma J$ and take its Langlands decomposition $g'=k'a'n'$ with respect to $P_0$; we get $||\Ad(g')|| \asymp ||\Ad(a')||$ and 
$$||\Ad(a')|| \ll \left(\frac{1}{\min_{1\le k\le r} \exp(\omega_k(\log(a')))}\right)^r.$$ 
By Remark \ref{Rmk:str_theory}, $\min_{\gamma \in \Gamma}{||\Ad(g\gamma)^{-1}||}^{-1}$ is comparable with ${||\Ad(g')^{-1}||}^{-1}$ and $h(x)$ is comparable with $h(g')$. Therefore,
$$\inj(x)^{-1} \ll h(x)^m $$

where $m := \frac{r}{\min_{1\le k \le r}\{\delta_*/q_k\}} = \frac{r}{\delta_*} \cdot \max_{1\le k \le r}\{q_k\}$. (Here, $\delta_*$ is from the definition of height function $h$, as in Remark \ref{Rmk:lowerbound_h}, and $\{q_k\}_{1 \le k \le r}$ are the positive real numbers used to define $h_k = \beta_k^{1/q_k}$.)
\end{proof}

Later in the proof of Proposition \ref{Prop:Number_of_Nearby_Sheets} and also in the construction of Margulis function $F_Y$ (see Theorem \ref{Margulis_inequality_F}), we shall see that Proposition \ref{Proposition:comparing_height_and_inj_radius} plays a key role, together with Theorem \ref{Thm:Margulis_height} and Corollary \ref{Corollary:Exponential_decay_height}. That is, the key property of the height function $h$ is that it is a Margulis function that is comparable with the injectivity radius. 

We now generalize our definition of height function $h:X \rightarrow [2, \infty)$ to the case when $\Gamma$ is cocompact, so that Theorem  \ref{Thm:Margulis_height}, Corollary \ref{Corollary:Exponential_decay_height}, and Proposition \ref{Proposition:comparing_height_and_inj_radius} all holds true also when $\Gamma$ is cocompact.

\begin{defn}\label{constant_function}
If $\Gamma$ is a cocompact lattice of $G$, then we define $h:X \rightarrow (0, \infty)$ to be the constant function,
$$h \equiv 2.$$
\end{defn}

\begin{proof}[Proof of Theorem \ref{Thm:Margulis_height} and Corollary \ref{Corollary:Exponential_decay_height} (for cocompact $\Gamma$)] All constant functions are Margulis functions. That is, for any $c>0$, $t>0$, and $x \in X$, we have
$$(A_{2,t}h)(x) := \frac{1}{m_U(B_2^U)}\int_{B_2 ^{U}}h(a_t ux)\,dm_U(u) = \frac{1}{m_U(B_2^U)}\int_{B_2 ^{U}} 2\,dm_U(u) = 2 < ch(x) + 2.$$
Corollary \ref{Corollary:Exponential_decay_height} can be proved in a similar way. Simply take $t_h=1$; for any $t \ge 2t_h$, we have $$(A_{t}h)(x) = 2 < \frac{1}{2^{t/t_h}}h(x) + 2.$$
\end{proof}

\begin{proof}[Proof of \ref{Proposition:comparing_height_and_inj_radius} (for cocompact $\Gamma$)] 
If $\Gamma$ is a cocompact lattice of $G$, then $X=G/\Gamma$ is compact and  $\varepsilon_0 := \inf_{x \in X} \inj(x) > 0$. Therefore, for any $x \in X$,
$$\inj(x)^{-1}\le \frac{1}{\varepsilon_0} = \frac{1}{2\varepsilon_0} h(x).$$
\end{proof}

For the remainder of the paper, the height function $h$ will refer to either of the two; if $\Gamma$ is non-uniform, then $h$ will refer to the function defined in Section 4 and if $\Gamma$ is cocompact, then $h$ will refer to the constant function in Definition \ref{constant_function} above.

\subsection{Return Lemma and Number of Sheets}

In Lemma \ref{Lem:Sheets_in_terms_of_height_function}, we shall first give a weaker bound on $\#I_Y(x)$, depending on both $\vol(Y)$ and $h(x)$. This shows that $\#I_Y(x)$ is uniformly bounded in terms of $\vol(Y)$ in the compact part of $X$.  In the case where $\Gamma$ is non-uniform, we shall make use of height function $h$ to analyze the cuspidal part of $X$. We shall show that if the sheets of $Y$ are very dense nearby a point $x$ in the cuspidal part of $X$, then for some $a_tux$ which lies in the compact part of $X$ the sheets of $Y$ must be very dense nearby $a_tux$ too. This gives us the desired result; even when $x$ is high up in the cuspidal part of $X$, there is a uniform bound for $\#I_Y(x)$, only in terms of $\vol(Y)$. Results in this section are analogous to Section 8 of Mohammadi--Oh [MO20] and Section 8 of Eskin--Mirzakhani--Mohammadi [EMM15].

 \begin{lem}[Return lemma]\label{lem:Return_lemma}
There exists a global constant $C_3=C_3(\dim(G)) >0 $ and $Q=Q(G/\Gamma, H) $ such that the following holds: for every $x\in X$, there exists $u\in B_1 ^U$ so that $a_{t_x}ux \in X_{\text{\rm{cpt}}} := \{x\in X: h(x) \le Q\}$  where $t_x=\constE\label{C:MI_for_height}\log(h(x))$. (constants $t_h$, $C_h$ and $B_h$ are as in Corollary \ref{Corollary:Exponential_decay_height}).
\end{lem}
\begin{proof}
By Lemma \ref{classifyH}, there exists $H' \in \mathcal H$ and $g \in G$ such that $H = gH'g^{-1}$. Let $\{a_t'\}$ denote the fixed one parameter subgroup of diagonalizable elements in $H'$ and let $U'$ be the unstable horospherical subgroup of $G$ with respect to $\{a_t'\}$. Note that 
$$\{a_t\} = g\{a_t'\}g^{-1} \text{\rm{ and }\it} B_1^U = gB_1^{U'}g^{-1}.$$
Let $t_h'$, $C_h'$, and $B_h'$ be the constants from Corollary \ref{Corollary:Exponential_decay_height}, with respect to $H' \subset G$. That is, $t_h'$, $C_h'$, and $B_h'$ are constants so that for all $t \ge t_h'$ and $x \in X$,
$$A_{t}'h(x) := \int _{B_1 ^{U'}} h(a_{t}'u' x)\,dm_{U'}(u') \le \frac{C_h'}{2^{t/t_h'}} \cdot h(x) +B_h'.$$

Take $C_3:=\frac{1}{\log2}\max_{H' \in \mathcal H} (t_h')$. Also, let $\sigma_g >1$ be the absolute constant following from the log continuity of $h$ so that for any $x \in X$, $\sigma_g^{-1}h(x) \le h(g^{\pm 1}x) \le \sigma_gh(x)$. (Here, $g$ is an element of $G$ so that $H=gH'g^{-1}$, as defined above.) Then, for every $x \in X$, 
\begin{align*}
(A_{t_{x}}'h)(g^{-1}x) &= \int _{B_1 ^{U'}} h(a_{t_x}'u' (g^{-1}x))\,dm_{U'}(u') \le \frac{C_h'}{2^{t_x/t_h'}} \cdot h(g^{-1}x) +B_h' \\
&\le \frac{C_h'}{2^{\frac{1}{\log2}t_h'\log(h(x))/t_h'}} \cdot \sigma_g h(x) +B_h' \le C_h'\sigma_g+B_h'
\end{align*}

and thus, there exists $u' \in B_1^{U'}$ such that $h(a_{t_{x}}' u' g^{-1}x) \le C_h'\sigma_g+B_h'$. 

Now, if we take $u := gu'g^{-1} (\in B_1^U)$, then we have that 
$$h(a_{t_{x}}ux) = h(ga_{t_{x}}'g^{-1} \cdot gu'g^{-1}x) = h(ga_{t_{x}}'u'g^{-1}x) \le \sigma_g h(a_{t_{x}}'u'g^{-1}x) = \sigma_g(C_h'\sigma_g+B_h').$$

Take $Q:= \sigma_g(C_h'\sigma_g+B_h')$. All that is left is to show that $C_3$ is a constant that only depends on $\dim(G)$ (and hence is independent of $H$ and lattice $\Gamma$). 

Recall that $t_h'$ (see the proof of Theorem \ref{Theorem:Margulis_Inequality_height} and also the proof of Corollary \ref{Corollary:Exponential_decay_height}) is the constant $t_{\delta_*, 1/4}' = t_{\delta_*, 1/4}'(G, H')$ as in Lemma \ref{Lem:LA_lemma}. From Lemma \ref{Lem:LA_lemma}, we have that $\delta_* = \frac{1}{2}\delta_1$ depends only on $\dim(G)$ and therefore, $t_{\delta_*, 1/4}'$ is a constant that only depends on $G$ and $H'$ (see Lemma \ref{Lem:LA_lemma}). 

Since $\mathcal H$ is finite data that only depend on $\dim(G)$ (Lemma \ref{classifyH}), $$C_3 =\frac{1}{\log2}\max_{H' \in \mathcal H} (t_h')$$ is a constant that only depends on $\dim(G)$. The fact that $C_3$ only depends on $\dim(G)$ will be later used to show that constant $D$ in Theorem \ref{MainThm} only depends on $\dim(G)$.

\end{proof}

\begin{remark} Let $\sigma_{U}>1$ be a global constant such that for all $v \in \Lie(G)$ and $u\in B_1 ^U$, 
$$ \sigma_{U}^{-1}||v|| \le ||u.v|| \le \sigma_{U}||v||$$
where $u.v$ denotes the adjoint representation of $u\in G$ on $v\in \Lie(G)$. For later computational purposes (see the proof of Proposition \ref{Prop:Number_of_Nearby_Sheets}), we shall assume that $Q > \sigma_U$. Even if we replace $Q= \sigma_g(C_h'\sigma_g+B_h')$ by $Q = \max\{\sigma_g(C_h'\sigma_g+B_h'), \sigma_U\}$, Lemma \ref{lem:Return_lemma} still holds.

We shall also assume that $C_3 \ge 1$ (this will also be used in the proof of Proposition \ref{Prop:Number_of_Nearby_Sheets}). Even if we replace $C_3=\frac{1}{\log2}\max_{H'\in \mathcal H} (t_h')$ with $C_3 = \max \{ \frac{1}{\log2}\max_{H' \in \mathcal H} (t_h'), 1\}$, Lemma \ref{lem:Return_lemma} still holds.
\end{remark}

\begin{remark}\label{epsilonX}
Note that by Proposition \ref{Proposition:comparing_height_and_inj_radius}, $X_{\text{\rm{cpt}}} \subseteq X_{\epsilon_X} = \{x \in X : \inj(x) \ge \epsilon_X \}$, where $\epsilon_X = \ref{C:comparing_height_and_inj_radius}^{-1}Q^{-m}$ (where $\ref{C:comparing_height_and_inj_radius}$ and $m$ as in Proposition \ref{Proposition:comparing_height_and_inj_radius}):
$$ \inj(a_{t_{x}}u x) \ge \ref{C:comparing_height_and_inj_radius}^{-1}h(x)^{-m} \ge \ref{C:comparing_height_and_inj_radius}^{-1}Q^{-m} = \epsilon_{X}.$$
\end{remark}

Recall, from the preliminaries that there exists an absolute constants $ \sigma_{0} >1 $ such that for all $ w \in \Lie(G)$ with $||w|| \le \epsilon_{X}$,
    $$ \sigma_{0}^{-1} ||w|| \le \dist(x, \exp(w)x) \le \sigma_{0}||w||.$$ 

\begin{defn}\label{Def:Window_I_Y}
Let $S$ be a subgroup with $H\subseteq S \subsetneq G$ and consider the decomposition of $\Lie(G)$ given by $\Lie(G)=\Lie(S)\oplus V_{S}$ where $V_S$ is $\Ad(\Lie(S))$-invariant, but not necessarily irreducible. For each closed $S$-orbit $Y=Sy$ and $x\in X$, we define the set
$$I_Y(x)=\{v\in V_{S}\setminus\{0\}:\|v\| \le \varepsilon_h \cdot h(x)^{-\kappa}, \ \exp(v)x \in Y\}$$
where global constants $\kappa \gg 1$ and $\varepsilon_h \ll 1$ are chosen to be; $\kappa := \max\{m, 3C_3\}$ and $\varepsilon_h = \min \{ 2^{\kappa}\epsilon_X, \frac{1}{2}\sigma_0^{-1}C_2^{-1}\}$ ($m$ and $C_2$ as in Proposition \ref{Proposition:comparing_height_and_inj_radius}, $C_3$ as in Lemma \ref{lem:Return_lemma}, and $\epsilon_X$ as in Remark \ref{epsilonX}).

\end{defn}

\begin{lem}\label{Lem:Sheets_in_terms_of_height_function}
Let $Y=Sy$ be a closed $S$-orbit where $H\subseteq S\subsetneq G$. For all $x\in X,$ we have that 
$$ \#I_{Y}(x) < C_{4} h(x)^{d_{S}m} \vol(Y)$$
where $d_S$ is the dimension of $\Lie(S)$ and $C_{4} = (4\ref{C:comparing_height_and_inj_radius})^{d_{S}}$.
\end{lem}
\begin{proof}
For any $x \in X$ and $v \in I_Y(x)$,
$$\dist(x,\exp(v)x) \le \sigma_{0}||v|| \le \sigma_0 \varepsilon_h h(x)^{-\kappa} \le \frac{1}{2} \ref{C:comparing_height_and_inj_radius}^{-1} h(x)^{-m} \le \frac{1}{2} \inj(x).$$
(Since $h \ge 2$ and $\varepsilon_h \le 2^\kappa \epsilon_X$, $\|v\| \le \varepsilon_h  h(x)^{-\kappa} \le \epsilon_X$ and thus $\dist(x, \exp(v)x) \le \sigma_0\|v\|$. Also, since $h \ge 1$, $\kappa \le m$, and $\varepsilon_h \le  \frac{1}{2}\sigma_0^{-1}C_2^{-1}$, we have $\sigma_0 \varepsilon_h h(x)^{-\kappa} \le \frac{1}{2} \ref{C:comparing_height_and_inj_radius}^{-1} h(x)^{-m}$.)

It follows that for each $v \in I_Y(x)$, $\inj(\exp(v)x)\ge \frac{1}{4} \inj(x)$, which means that the balls 
$$\left(B_{Y}(\exp(v)x,\inj(x)/4)\right)_{v \in I_{Y}(x)}$$ 
are disjoint from each other. Hence, 
\begin{align*}
        \#I_{Y}(x) \cdot \vol(B_{S}(e, \inj(x)/4) &= \vol\{\cup (B_{Y}(\exp(v)x, \inj(x)/4)) : v \in I_{Y}(x)\} \\
         &\le  \vol(B_{Y}(x, \inj(x)) \le \vol(Y).
\end{align*}
Therefore, 
\begin{align*}
        \#I_{Y}(x) &\le \vol(B_{S}(e, \inj(x)/4))^{-1} \cdot \vol(Y) \\
         &< 4^{d_{S}}\inj(x)^{-d_{S}} \cdot \vol(Y) \le 4^{d_{S}}\ref{C:comparing_height_and_inj_radius}^{d_{S}}h(x)^{d_{S}m} \vol(Y).
\end{align*}
\end{proof}

\begin{proof}[Proof of Proposition \ref{Prop:Number_of_Nearby_Sheets}]
If $x \in X$ is so that $h(x) \le Q^2$, then by Lemma \ref{Lem:Sheets_in_terms_of_height_function}, 
$$ \#I_{Y}(x) < (4C_2)^{d_S} Q^{2d_{S}m} \vol(Y) .$$
Suppose that $h(x) > Q^2$. By Lemma \ref{lem:Return_lemma}, there exists $u\in B_1 ^U$ and $t_{x} = \ref{C:MI_for_height}\log(h(x))$ such that 
$h(a_{t_{x}}ux) \le Q.$

We claim that if $v \in I_{Y}(x)$, then $a_{t_{x}}u.v \in I_{Y}(a_{t_{x}}ux)$ and moreover, the map $a_{t_{x}}u:I_{Y}(x) \rightarrow I_{Y}(a_{t_{x}}ux)$ which sends $v \mapsto a_{t_{x}}u.v$ is injective. If $v \in I_{Y}(x)$, then
\begin{align*}
        ||a_{t_{x}}u.v || & \le e^{t_{x}}\sigma_{U}||v|| = \exp(\ref{C:MI_for_height}\log(h(x)))\sigma_{U}||v|| = \sigma_{U}h(x)^{C_{3}}||v|| \\
         & \le \sigma_{U}h(x)^{C_{3}}\varepsilon_h h(x)^{-\kappa}  \le \sigma_{U} \varepsilon_h h(x)^{-2\kappa/3} \le \sigma_{U} \varepsilon_h Q^{2 \cdot (-2\kappa/3)} = \sigma_{U} \varepsilon_h Q^{-4\kappa/3} \le \varepsilon_h Q^{-\kappa}
\end{align*}
since $\kappa \ge 3C_3 \ge 3$, $h(x) > Q^2$, and $Q \ge \sigma_U$.
On the other hand, $\ h(a_{t_{x}}ux)^{-\kappa} \ge Q^{-\kappa}$ and thus, 
$$||a_{t_{x}}u.v || < \varepsilon_h h(a_{t_{x}}ux)^{-\kappa}.$$
Moreover,
$$a_{t_{x}}\exp(v)x = \exp(a_{t_{x}}u.v)a_{t_{x}}ux \in Y$$ since $Y$ is $S$-invariant. Therefore, $a_{t_{x}}u.v \in I_{Y}(a_{t_{x}}u_{s}x)$ and the map $v \mapsto a_{t_{x}}u.v$ from $I_{Y}(x)$ to $I_{Y}(a_{t_{x}}ux)$ is injective. Consequently,
$$ \#I_{Y}(x) \le \#I_{Y}(a_{t_{x}}ux) <(4C_2)^{d_S} Q^{2d_{S}m} \vol(Y) .$$

We complete the proposition by taking $C_{1}=(4C_2)^{d_G}Q^{2d_Gm}$ where $d_G$ denotes the dimension of $\Lie(G)$.
\end{proof}

We end this section with the following lemma on $\kappa$, which will be later used to show that constant $D$ in Theorem \ref{MainThm} is an absolute constant that only depends on $\dim(G)$.

\begin{lem}\label{kappa} Constant $\kappa$ in Definition \ref{Def:Window_I_Y} has an upper bound as a function of $\dim(G)$.
\end{lem} 
\begin{proof}
Since $\kappa := \max\{m, 3C_3\}$ and $C_3$ is a constant that only depends on $\dim(G)$ (see Lemma \ref{lem:Return_lemma}), it is enough to show that $m$ from Proposition \ref{Proposition:comparing_height_and_inj_radius} has an upperbound as a function of $\dim(G)$.

Recall that (see the proof of Proposition \ref{Proposition:comparing_height_and_inj_radius}) $$m = \frac{r}{\delta_*} \cdot \max_{1\le k \le r}\{q_k\} = \frac{r}{\frac{1}{2}\delta_1} \cdot \max_{1\le k \le r}\{q_k\}$$
where $\delta_1$ is as in Lemma \ref{Lem:LA_lemma} and $q_k$ are as in definition of $h_k$. By Lemma \ref{Lem:LA_lemma}, $\delta_1$ is a constant that only depends on $\dim(G)$. From reduction theory, we have that $r \le \dim(G)$. Lastly, since the root system $\{\omega_k\}$ is a data that only depends on $G$, the value of $\max_{1\le k \le r}\{q_k\}$ also has an upper bound in terms of $\dim(G)$.

\end{proof}

\section{Margulis function: Construction and estimates}
In this section we construct Margulis functions that we associate to the orbit of an intermediate subgroup. The main result of this section is a Margulis inequality for the functions we consider (Theorem \ref{Margulis_inequality_F}).

\begin{defn}[Margulis function]\label{def:Margulis_function}
 For an intermediate subgroup $H\subseteq S\subsetneq G$ and closed $S$-orbit $Y=Sy$, define $f_Y:=X\to (0,\infty)$ by 
$$f_Y(x):=\begin{cases}
  \sum_{v\in I_Y(x)} \|v\|^{-\delta_F},  & \text{ if } I_Y(x)\ne \emptyset \\
  h(x), & \text{ otherwise}
\end{cases}.$$
where $\delta_F := \min\{\delta_0/2, 1/\kappa\}$, $\delta_0$ as in Lemma \ref{Lem:General_LA_lemma} applied to the adjoint representation of $G$ on the Lie algebra $\Lie(G)$ and $\kappa$ as in Definition \ref{Def:Window_I_Y}.

For $\lambda \ge 1$, define $F_{\lambda,Y}:X\to(0,\infty)$ by
$$F_{\lambda, Y}(x)=f_Y(x)+\lambda \vol(Y)h(x).$$
\end{defn}

We will later fix an explicit $\lambda$ in Theorem \ref{Margulis_inequality_F} so that $F_{\lambda, Y}$ satisfies a Margulis inequality.

\begin{remark}\label{dimG}
Note that by Lemma \ref{Lem:General_LA_lemma} and Lemma \ref{kappa}, $\delta_F$ can be thought of as a constant that only depends on $\dim(G)$. Later, this will imply that constant $D$ in Theorem \ref{MainThm} is only dependent on $\dim(G)$.
\end{remark}

\begin{prop}[Log continuity of $F_{\lambda, Y}$]\label{Proposition:Log_continuity_F}
 Let $K$ be a compact subset of $G$. Then there exists an absolute constant $\sigma=\sigma(K)$ (only depending on $K$ and independent on the choice of $Y$ and $\lambda$) such that for all Margulis function $F_{\lambda, Y}$, point $x\in X$, and $g \in K$,
$$\sigma^{-1}F_{\lambda, Y}(x)\le F_{\lambda, Y}(gx)\le \sigma F_{\lambda, Y}(x).$$
\end{prop}

\begin{proof}
Recall that we denote the adjoint representation of elements $g\in G$ and $v\in \Lie(G)$ as $g.v$.
 Since $K\subset G$ is compact, there exists $R_K\ge1$ so that 
 $$R_K ^{-1}\|v\|\le\|g.v\|\le R_K \|v\|$$
 for every $g\in K$ and $v\in \Lie(G)$.
 Also, by the log continuity of $h$, there exists $\sigma_h = \sigma_h(K)\ge1$ so that 
 $$\sigma_h ^{-1} \cdot h(x)\le h(gx)\le \sigma_h \cdot h(x)$$
 for every $g\in K$ and $x\in X$.
 
If $I_Y(gx)$ is empty, then
$$f_Y(gx)= h(gx)\le \sigma_h \cdot h(x)$$

Now suppose that $I_Y(gx)$ is not empty. Set $\varepsilon = R_K^{-1} \varepsilon_h h(x)^{-\kappa}$.
Note that we can write
$$f_Y(gx)=\sum_{v\in I_Y(gx),\|v\|<\varepsilon}\|v\|^{-\delta_F} + \sum_{v\in I_Y(gx),\|v\|\ge \varepsilon}\|v\|^{-\delta_F}.$$

By Proposition \ref{Prop:Number_of_Nearby_Sheets} since $\#I_Y(gx)\le C_1\vol(Y)$, then we have the following bound for the second term above,
$$ \sum_{v\in I_Y(gx),\|v\|\ge \varepsilon}\|v\|^{-\delta_F}\le C_1\vol(Y) \varepsilon^{-\delta_F}= (C_1(R_K\varepsilon_h^{-1}) ^{\delta_F}\vol(Y)) \cdot h(x)^{\delta_F \kappa} \le C_5\vol(Y) h(x)$$

where $C_5 := C_1(R_K\varepsilon_h^{-1})^{\delta_F}$. (Here we are using that $h\ge 1$ and $\delta_F \le 1/\kappa$.)
If there is no $v\in I_Y(gx)$ with $\|v\|<\varepsilon$, then this proves the claim. If there is $v\in I_Y(gx)$ with $\|v\|<\varepsilon$, then
$$\|g^{-1}.v\| < R_K\varepsilon= \varepsilon_h h(x)^{-\kappa}.$$
Thus, $g^{-1}.v\in I_Y(x)$. Setting $v'=g^{-1}v$ yields that
$$\sum_{v\in I_Y(gx),\|v\|<\varepsilon}\|v\|^{-\delta_F} \le \sum_{v'\in I_Y(x)}\| g.v'\|^{-\delta_F} \le R_K ^{\delta_F}\sum_{v'\in I_Y(x)}\|v'\|^{-\delta_F} =R_K^{\delta_F}f_Y(x).$$

In total, we have a bound of the form
$$ f_Y(gx)\le R_K^{\delta_F}f_Y(x) + (C_5\vol(Y)+\sigma_h)h(x) .$$
Thus, 
\begin{align*}
    F_{\lambda, Y}(gx) &= f_Y(gx)+\lambda \vol(Y) h(gx)\\
    &\le R_K^{\delta_F}f_Y(x) + (C_5\vol(Y)+\sigma_h)h(x)+ \lambda \vol(Y)\cdot \sigma_h h(x).
\end{align*}
Note that $\vol(Y)$ is bounded away from zero. That is, there exists an absolute constant $ \tau >0 $ such that for any intermediate subgroup $H \subseteq S \subsetneq G$  and closed $S$-orbit $Y=Sy$, $\vol(Y)>\tau$. (See Lemma \ref{Lem:Lower_bound_on_volume} below.) Since $\lambda \ge 1$, we have 
\begin{align*}
    F_{\lambda, Y}(gx) &\le R_K^{\delta_F}f_Y(x) + (C_5+\sigma_h\tau^{-1} +\sigma_h\lambda)\vol(Y)h(x)\\
    &\le R_K^{\delta_F}f_Y(x) + (C_5+\sigma_h\tau^{-1} +\sigma_h)\lambda\vol(Y)h(x).
\end{align*}
Put $\sigma:= \max \{R_K^{\delta_F}, C_5+\sigma_h(\tau^{-1}+1)\}$. We note that $\sigma$ is independent of $Y$ and $\lambda$.

The remaining inequality $$\sigma^{-1}F_{\lambda, Y}(x)\le F_{\lambda, Y}(gx)$$ is proved in a similar fashion.
\end{proof}

\begin{lem}\label{Lem:Lower_bound_on_volume}
There exists an absolute constant $\tau>0$ such that for any intermediate subgroup $H\subseteq S\subsetneq G$ and closed $S$-orbit $Y=Sy$, $\vol{(Y)}>\tau$.
\end{lem}
\begin{proof} This follows from the quantitative non-divergence of Dani--Margulis \cite{MR1101994}. Let $U'$ be a 1-parameter unipotent subgroup of $H$. By the quantitative 
non-divergence of the action of $U'$ on $X$, there exists some $\rho >0$ such that $m_Y(X\setminus X_{\rho })<0.01$
for every closed $S$-orbit $Y=Sy$ ($H \subseteq S\subsetneq G$ and $y\in Y$), where $m_Y$ is the probability Haar measure on $Y$.

Note that we have the Lie algebra decomposition $\Lie(G)=\Lie(S)\oplus V_{S}$.
Let $\eta\asymp \rho$ be so that the map $g\mapsto gx$ is injective for all $x\in X_\rho$ and all
$$g\in \text{\rm{Box}}(\eta):= \exp(B_{\eta}^{\Lie(S)} )\exp(B_{\eta}^{V_{S}}).$$

For any connected component $C$ of $Y\cap \text{\rm{Box}}(\eta)z$ with $z \in X_\rho$, there exists some $v\in V_{S}$ such that
$$C= C_v:=\exp(B_{\eta}^{\Lie(S)})\exp(v)z.$$

Since $m_Y(X\setminus X_{\rho})<0.01$, $Y \cap X_{\rho} \neq \emptyset$ and thus, 
$$\vol(Y) \ge \eta^{d_S} \ge \eta^{d_H}.$$

\end{proof}

\begin{thm}[Margulis Inequality for $F_Y$]\label{Margulis_inequality_F}
 Let $H\subseteq S\subsetneq G$ be an intermediate subgroup and  $Y=Sy$ a closed $S$-orbit. Let $F_{\lambda, Y} \,(\lambda \ge 1)$ denote the Margulis functions associated to $Y$. For any $0<c<1$, there exists $t=t_{F,c}>0$ such that there exists global constants $\lambda_1 \ge 1$ and $E_1>0$ such that the following holds for any closed orbit $Y=Sx\,(H \subseteq S \subsetneq G)$ and its corresponding Margulis function $F_Y :=F_{\lambda_1,Y}$ :
 $$A_{2,t}F_{Y} \le c F_{Y} + E_1\vol(Y).$$
\end{thm}

\begin{proof}
 We have that $\Lie(G)=\Lie(S)\oplus V_S$ where $V_S$ is $\Ad(H)$-invariant, but has no $\Ad(H)$-invariant vectors. Since $\delta_F \le \delta_0/2 < \delta_0$, by the Linear Algebra Lemma (Lemma \ref{Lem:General_LA_lemma}) there exists $t_c '>0$ so that 
 $$\frac{1}{m_U(B_2^U)}\int_{B_2 ^{U}}\frac{1}{\|a_{ t} u.v\|^{\delta_F}}\,dm_U(u)< \frac{c}{\|v\|^{\delta_F}}$$
for every $v\in V_S$ and $t\ge t_c'$. 

On the other hand, by Theorem \ref{Thm:Margulis_height} there exists $t_c''>0$ such that for all $t \ge t_c''$, there exists absolute constant $B_t$ (depending on $t$) so that
\begin{equation}
A_{2,t}h \le \frac{c}{2}h+B_t.\end{equation}

Let $t_c'$ and $t_c''$ be as above and take $t = t_{F,c} = \max\{t_c', t_c''\}$.

We first find a bound on $A_{t}f_Y$. Then we pick a particular value for $\lambda$ and  combine the first bound with the bound from Theorem \ref{Thm:Margulis_height} to reach the desired bound on $F_Y$.

Fix compact set $K_t=a_{t}B_1^{U}$ and let $R_{K_t}\ge 1$ be a constant so that
$$R^{-1}_{K_t}\|v\|\le \|g.v\|\le R_{K_t} \|v\|$$
for every $g\in K_t$ and $v\in V_S$.

If $I_Y(gx)$ is empty for every $g \in K_t$, then $F_Y$ is just a constant multiple of the height function $h$ and thus by (1), 
\begin{align*}
A_{2,t}F_Y &= A_{2,t} (\lambda \vol(Y)h) \le \lambda \vol(Y) \left(\frac{c}{2}h+B_{t}\right) \\ 
&\le c (\lambda \vol(Y) h)+ \lambda \vol(Y) B_{t} \le cF_Y+ \lambda B_t \vol(Y).
\end{align*}

Suppose that $I_Y(gx)$ is not empty. Set $\varepsilon = R_{K_t}^{-1} \varepsilon_h h(x)^{-\kappa}$.
Note that we can write
$$f_Y(gx)=\sum_{v\in I_Y(gx),\|v\|<\varepsilon}\|v\|^{-\delta_F} + \sum_{v\in I_Y(gx),\|v\|\ge \varepsilon}\|v\|^{-\delta_F}.$$

and by the calculation from Proposition \ref{Proposition:Log_continuity_F}, we obtain a bound of the form,
$$ f_Y(gx)\le \sum_{v'\in I_Y(x)}\| g.v'\|^{-\delta_F} + (C_5\vol(Y)+ \sigma_h)h(x) $$
where $C_5 = C_1 (R_{K_t}\varepsilon_h^{-1})^{\delta_F}$ and $\sigma_h = \sigma_h(K_t)$ is the constant from the log continuity property of $h$. Integrating over $B_2 ^{U}$ yields
\begin{align*}
    &\frac{1}{m_U(B_2^U)}\int_{B_2 ^{U}}f_Y(a_tux)\,dm_U(u) \\
    &\le \sum_{v'\in I_Y(x)}\frac{1}{m_U(B_2^U)}\int_{B_2 ^U}\| a_t u.v'\|^{-\delta_F}\,dm_U(u) + (C_5\vol(Y)+\sigma_h)h(x)
\end{align*}
and since $t\ge t_c'$ by Lemma \ref{Lem:General_LA_lemma}, $$\sum_{v'\in I_Y(x)}\frac{1}{m_U(B_2^U)}\int_{B_2 ^{U}}\| a_t u.v'\|^{-\delta_F}\,dm_U(u) < c\sum_{v'\in I_Y(x)} \| v\|^{-\delta_F}.$$

To summarize, we get the bound
\begin{equation}
\label{eqn:f_Ybound}
\frac{1}{m_U(B_2^U)}\int_{B_2 ^{U}}f_Y(a_tux)\,dm_U(u)\le cf_Y(x) +(C_5+\sigma_h \tau^{-1})\text{vol}(Y) h(x)\end{equation}
where $\tau$ is as in Lemma \ref{Lem:Lower_bound_on_volume}.

Combining equation (1) and equation (\ref{eqn:f_Ybound}) together we have
 \begin{align*}
  A_{2,t} F_Y &\le c \cdot f_Y +(C_5+\sigma_h\tau^{-1})\vol(Y) \cdot h +\lambda \vol(Y) \left(\frac{c}{2} \cdot h+ B_t \right)\\
  &=c \cdot f_Y + \left(C_5+\sigma_h\tau^{-1}+ \frac{\lambda c}{2}  \right) \vol(Y) \cdot h +\lambda B_t \vol(Y).
 \end{align*}

Now choose $\lambda_1:= 2(C_5+\sigma_h\tau^{-1})/c$ so that $\left(C_5+\sigma_h\tau^{-1}+ \frac{\lambda_1 c}{2} \right) = \lambda_1 c$ and we get the desired result,
$$A_{2,t}F_Y \le cF_Y + E_1 \vol(Y)$$

where $E_1 =  \lambda_1 B_t = 2(C_5+\sigma_h\tau^{-1})B_t/c$. Note that $\lambda_1 \ge 1$, since $C_5$ and $\sigma_h$ are constants larger than 1 and $\tau$ and $c$ are constants smaller than 1. Also note that both $\lambda_1$ and $E_1$ are constants independent of $Y$.
\end{proof}

For the remainder of the paper, $F_Y$ will refer to the Margulis function $F_{\lambda_1,Y}$ with fixed $\lambda_1 = 2(C_5+\sigma_h\tau^{-1})/c$ as in Theorem \ref{Margulis_inequality_F}.

\begin{cor}[Exponential decay]\label{Cor:exponential_decay_F}
 There exists global constants $C_F>0$ and $E_2>0$ such that for any closed orbit $Y$ and for any $t \ge t_F$ (where $t_F := t_{F,1/2}$ defined as in Theorem \ref{Margulis_inequality_F}),
$$(A_{t} F_Y)(x)\le \frac{C_F}{2^{t/t_F}}F_Y(x)+E_2\vol(Y)$$ 
\end{cor}

\begin{proof}
The result follows from Theorem \ref{Thm:Abstract_exp_decay}, Proposition \ref{Proposition:Log_continuity_F}, and Theorem \ref{Margulis_inequality_F}. Especially, the fact that $C_F$ and $E_2$ are global constants, independent of $Y$ follows from the fact that the log continuity constants for $F_Y$ depends only on the compact set and is independent on the choice of $Y$ (see Proposition \ref{Proposition:Log_continuity_F}).
\end{proof}

The following is following result is standard, see \cite[Lemma 7.3]{Isolations}  or \cite[Lemma 11.1]{MR3418528}.
\begin{prop}[Margulis function on average]
 Let $H\subseteq S\subsetneq G$ denote an intermediate subgroup and $Y=Sy$ be a closed $S$-orbit. Let $F_Y$ denote the associated Margulis function from Theorem \ref{Margulis_inequality_F}. Let $\mu$ be an $A$-ergodic $U$-invariant measure with $\mu(Y)=0$. Then $$F_Y\in L^1(\mu).$$
\end{prop}
\begin{proof}
In this proof we will drop the subscript $Y$ in $F_Y$ for simplicity.
For $k\in \mathbb N$, let $F_k:=\min (F,k)$. Take $t$ to be $t_F$, the constant obtained from Corollary \ref{Cor:exponential_decay_F}.
 
By Moore's ergodicity theorem, we have that the action of $A=\{a_t:t\in \mathbb R\}$ is ergodic $X$. Then, by the Birkhoff ergodic theorem, for $\mu$-a.e. $x\in X$ and $k\in \mathbb N$, 
 $$\lim_N\frac{1}{N}\sum_{n=1} ^N F_k(a_{nt}x)=\int F_k \,d\mu.$$

There exists some $x_0\in X$ such that for $m_U$-a.e. $u\in B_1 ^U$, 
 $$\lim_N\frac{1}{N}\sum_{n=1} ^N F_k(a_{nt}ux_0)=\int F_k \,d\mu.$$
 
Thus, by Egoroff's theorem, for each $k\in \mathbb N$ there exists a subset $E_k\subseteq B_1 ^U$ with $m_U(E_k)>\frac{1}{2}$ and $N_k\in \mathbb N$ such that for every $N>N_k$ and $u\in E_k$, 
$$\frac{1}{N}\sum_{n=1} ^N F_k(a_{nt}ux_0)>\frac{1}{2}\int F_k \,d\mu.$$

Integrate this inequality over $B_1 ^U$ to obtain
$$\frac{1}{N}\sum_{n=1} ^N \int_{B_1 ^U}F_k(a_{nt}ux_0)dm_U(u)> \frac{1}{2}\int F_k \,d\mu.$$

By Corollary \ref{Cor:exponential_decay_F}, for all $n \in \mathbb{N}$,
$$\int_{B_1 ^U}F_k(a_{nt}ux_0)dm_U(u)\le \int_{B_1 ^U}F(a_{nt}ux_0)dm_U(u)<\frac{C}{2^n}F(x_0)+b$$
where $C=C_F$ and $b =E_2\vol(Y)$.

Choose $n_0$ so that $\frac{1}{2^{n_0}}F(x_0)\le 1$. Then for each $n\ge n_0$ and $N>\max(N_k,kn_0)$, 
\begin{align*}
    \frac{1}{2}\int F_k \,d\mu&<\frac{1}{N}\sum_{n=1} ^N \int_{B_1 ^U}F_k(a_{nt}ux_0)dm_U(u)\\
    &=\frac{1}{N}\sum_{n=1} ^{n_0} \int_{B_1 ^U}F_k(a_{nt}ux_0)dm_U(u)+\frac{1}{N}\sum_{n=n_0+1} ^N \int_{B_1 ^U}F_k(a_{nt}ux_0)dm_U(u)\\
    &\le\frac{n_0k}{N} + \frac{1}{N}\sum_{n=n_0+1} ^N (\frac{C}{2^n} F(x_0)+b)\\
    &\le 1 + \frac{1}{N}\sum_{n=n_0+1} ^N (C+b) = C+b+1.
\end{align*}

Thus, 
$$\int F_k\,d\mu\le 2(C+b+1).$$
Taking $k\to\infty$ and using the monotone convergence theorem, we have $F\in L^1(\mu)$.
 
\end{proof}

\section{Isolation of closed orbits}

In this section, we prove Theorem \ref{Thm:Isolation_in_distance} and Theorem \ref{finiteness_closed_orbits_bounded_volume}. Results in this section are analogous to Section 10 of \cite{Isolations}.

\begin{proof}[Proof of Theorem \ref{Thm:Isolation_in_distance}]
We shall prove the following: for any two distinct closed $S$-orbits $Y=Sy$ and $Z=Sz$ ($H \subseteq S \subsetneq G$) of finite volume,
$$\dist(Y \cap K, Z)\gg_K \vol(Y)^{-1/\delta_F}\vol(Z)^{-1/\delta_F}$$
where $K$ is a compact subset of $X$ and $\delta_F$ is as in Definition \ref{def:Margulis_function}. Recall that $\delta_F$ is a global constant only depending on $G$ and $H$ (and thus, independent of the choice of $\Gamma$, see Remark \ref{dimG}).

Let $m_Y$ denote the Haar probability measure on $Y$. Since $m_Y$ is an $A$-ergodic $S$-invariant probability measure, $m_Y(A_tF_Z) = m_Y(F_Z)$. Thus,
by integrating the Margulis inequality $$ A_t(F_Z) < c F_Z + E_2\vol(Z)$$ ($c$ is some positive constant smaller than 1) from Corollary \ref{Cor:exponential_decay_F} over $Y$, we get $$m_Y(F_Z) \le \frac{E_2}{1-c} \vol(Z).$$ 
Since $K$ is compact, by log continuity of $F_Z$, there exists $\sigma = \sigma_{F_Z}(K) >1$ such that for any $x \in X$ and $g \in B_{\epsilon}^G$,
$$F_Z(x) \le \sigma F_Z(gx).$$
Recall that the log continuity coefficients for $F_Z$ is independent of closed orbit $Z$ that only depends on the compact set that $g$ belongs to (Proposition \ref{Proposition:Log_continuity_F}); $\sigma$ is a global constant that only depends on $K$. 

Now since $K$ is compact, $\epsilon = \epsilon_K :=\min_{x \in K}\inj(x) >0$. For any point $y \in Y \cap K$,
\begin{align*}
    f_Z(y) \le F_Z(y) &\le \frac{1}{m_Y(B_{\epsilon}^S(y))} \int_{g \in B_{\epsilon}^S(e)} \sigma F_Z(gy)dm_Y(gy)\\ 
    &\le \frac{1}{m_Y(B_{\epsilon}^S(y))}\sigma m_Y(F_Z) \le \frac{\sigma}{\epsilon^{d_S}}\cdot \frac{E_2}{1-c} \vol(Y)\vol(Z).
\end{align*}
Lastly, we observe that $\dist(y, Z)^{-\delta_F} \ll  f_Z(y)$. If $I_Z(y)$ is non-empty, then $$f_Z(y) =\sum_{v\in I_Z(y)} \|v\|^{-\delta_F} \ge \dist(y, Z)^{-\delta_F}.$$ If $I_Z(y)$ is empty, then $\dist(y, Z) > \varepsilon_h h(y)^{-\kappa}$ and so $$f_Z(y)=h(y) \ge h(y)^{\kappa \delta_F} \ge \varepsilon_h^{\delta_F} \dist(y, Z)^{-\delta_F}.$$ (Here, we are using that $h \ge 1$ and $\delta_F:=\min\{\delta_0/2, 1/\kappa\} \le 1/\kappa$.)
\end{proof}

\begin{proof}[Proof of Theorem \ref{finiteness_closed_orbits_bounded_volume}] 
We shall prove the following: there exists a global constant $C_6>0$ such that for any 
intermediate subgroup $H\subseteq S\subsetneq G$,
$$\#\{Y: Y=Sy\text{\rm{ is a closed }\it}S\text{\rm{-orbit and }\it}\vol(Y)\le R\}<C_6 R^{d_G/\delta_F}$$
where $d_G$ is the dimension of $\Lie(G)$ and $\delta_F$ is as in Definition \ref{def:Margulis_function}.

We define constants $\rho >0$ and $\eta >0$ as in Lemma \ref{Lem:Lower_bound_on_volume}. Let $\rho>0$ be a constant such that $m_Y(X\setminus X_{\rho})<0.01$ for every closed $S$-orbit $Y=Sy$ ($H\subseteq S\subsetneq G$ and $y\in Y$) and let $\eta\asymp \rho$ be a constant so that the map $g\mapsto gx$ is injective for all $x\in X_\tau$ and all
$$g\in \text{\rm{Box}}(\eta):= \exp(B_{\eta}^{\Lie(S)} )\exp(B_{\eta}^{V_{S}})$$

(here, $\Lie(G) = \Lie(S) \oplus V_S$ is the Lie algebra decomposition).
Then, for any connected component $C$ of $Y \cap \text{\rm{Box}}(\eta)z$ with $z \in X_\rho$, there exists some $v\in V_{S}$ such that
$$C= C_v:=\exp(B_{\eta}^{\Lie(S)})\exp(v)z.$$
For $R>0$, let
$$\mathcal Y(R)=\{Y: Y=Sy\text{ is closed } S \text{-orbit and }R/2<\text{Vol}(Sy)\le R\}.$$

By Theorem \ref{Thm:Isolation_in_distance}, for any distinct connected components $C_v$ and $C_{v'}$ in $\mathcal Y(2^k) \cap \text{\rm{Box}}(\eta)z$, we have that
$$\|v-v'\|\gg_\rho 2^{-2k/\delta_F}.$$

The cardinality of any $2^{-2k/\delta_F}$-separated set in $B_{\eta}^{V_S}$ is up to multiplicative constant, 
$$(2^{2k/\delta_F})^{(d_G-d_S)}.$$

Since $\vol(\text{\rm{Box}}(\eta))=\eta ^{d_G},$ we can cover $X_\rho$ by $M=O(\eta^{-d_G})$ many sets of the form $\text{\rm{Box}}(\eta)z$. Choose such a cover
$\{\text{\rm{Box}}(\eta)z_j:j=1,...,M\}$.

Then,
\begin{align*}
        \#\mathcal Y(2^k)&\le 2^{-k+1}\sum_{Y \in \mathcal Y(2^k)} \vol(Y)\\
          &\ll2^{-k+1} \sum_{j=1} ^M \sum_{C_v\in \text{\rm{Box}}(\eta)z_j}\vol(C_v).
\end{align*}
Since (1) $\vol(C_v) = \eta^{d_S} \ll 1$ for each $C_v$, (2) $\#\{C_v \in \text{\rm{Box}}(\eta)z_j \}\ll (2^{2k/\delta_F})^{(d_G-d_S)}$ for each $\text{\rm{Box}}(\eta)z_j$, and (3) $M = O(1)$, we have
$$2^{-k+1} \sum_{j=1} ^M \sum_{C_v\in \text{\rm{Box}}(\eta)z_j} \vol(C_v)\ll2^{2k(d_G-d_S)/\delta_F-k+1}.$$

Recall that $\vol(Sy)\ge \eta^{d_S}$ since the volume of the orbit needs to contain at least one connected component $C_v$ in some $\text{\rm{Box}}(\eta)z_j$. Let $n_0=\lfloor d_S \log_2(\eta)\rfloor$ and $n_R=\lceil \log_2(R)\rceil$.
Since 
$$\{Sy: Sy \text{ is closed and }\vol(Sy)\le R\}\subseteq \bigcup_{k=n_0} ^{n_R}\mathcal Y(2^k),$$
we get
$$\#\{Y: Y=Sy\text{\rm{ is a closed }\it}S\text{\rm{-orbit and }\it}\vol(Y)\le R\} \le \sum_{k=n_0} ^{n_R}\#\mathcal Y(2^k)\ll R^{{d_G}/\delta_F}.$$
\end{proof}

\section{Proof of the Main Theorem}
In this section we prove Theorem \ref{MainThm}.

\begin{proof}[Proof of Theorem \ref{MainThm}]
For each point $x \in X$, our choice of $T_x$ will be $T_x := h(x)^{1/\delta_F}$. Note that $h(x)$ is bounded in the compact part of $X$ and thus, $T_x$ can be chosen uniformly within a compact subset of $X$.

We fix a point $x \in X$. Let $(T, R)$ be a pair of real numbers such that $T>T_x$ and $R>2$ and suppose that $x$ is  $(R, 1/T)$-Diophantine with respect to $H$. Our final goal will be to show that there exist absolute constants $D=D(\dim(G))$, $A=A(G/\Gamma, H)$, and $C=C(G/\Gamma, H)$ independent of $x$, $R$, and $T$, such that condition (2) of Theorem \ref{MainThm} holds: for all $t \ge A\log T$, $$m_{U}\left(\left\{u\in B_1 ^U: a_{t}ux \ \text{is not $(R, R^{-D}, R^{-D})$-Diophantine} \right\}\right)< CR^{-1}.$$

Recall that $x' \in X$ is $(R,R^{-D}, R^{-D})$-Diophantine with respect to $H$ if and only if 
\begin{itemize}
    \item[(1)] $\inj(x) \ge R^{-D}$ and
    \item[(2)] for all intermediate subgroups $H\subseteq S\subsetneq G$ and all closed $S$-orbit $Y=Sx'$ with $\vol (Y)\le R$, we have $\dist(x,Y) \ge R^{-D}$.
\end{itemize}

\textbf{Step 1: Recurrence to the compact part.}

First, we show that there exists $D_1=D_1(\dim(G)) >0$ and $A_1=A_1(G/\Gamma, H) >0$ such that for all $D \ge D_1$ and $A \ge A_1$, the following is true: for all $t \ge A \log T$,
$$m_U(\{u\in B_1 ^U: \inj(a_{t}ux) < R^{-D} \}) \ll R^{-1}.$$

Take $A_1 = \delta_F t_h/\log 2$. Then, for any $t \ge A \log T$,
$$t \ge A \log T \ge A_1 \log T_x = A_1 \log (h(x)^{1/\delta_F}) \ge A_1/\delta_F \cdot \log 2 = t_h$$ and so by Corollary \ref{Corollary:Exponential_decay_height},
\begin{align*}
(A_{t}h)(x) & =\int _{B_1 ^U} h(a_{t}u x)\,dm_U(u) \le \frac{C_h}{2^{t/t_{h}}} h(x)+B_h \\
& \le \frac{C_h}{2^{A \log T/t_{h}}} h(x)+B_h \\
& \le \frac{C_h}{2^{A_1 \log T_x/t_{h}}} h(x)+B_h \\
& = \frac{C_h}{2^{(A_1/\delta_Ft_{h}) \cdot \log(h(x))}} h(x)+B_h  = C_h+B_h.
\end{align*}

For $x' \in X$, if $\inj(x') \le R^{-D}$, then 
by Proposition \ref{Proposition:comparing_height_and_inj_radius},
$$h(x') \ge C_2^{-1/m} \cdot \inj(x)^{-1/m} > C_2^{-1/m} \cdot R^{D/m}.$$
Applying this observation to points $a_{t}ux$ and using Chebyshev's theorem we obtain,
\begin{align*}
&m_U(\{u\in B_1 ^U: \inj(a_{t}ux) \le R^{-D}\}) \\
& \le m_U(\{u\in B_1 ^U: h(a_{t}ux) > C_2^{-1/m} \cdot R^{D/m}\}) \\
& < C_2^{1/m} R^{-D/m}\cdot (A_{t}h)(x) \\
& \le C_7 R^{-D/m}
\end{align*}
where $C_7 := C_2^{1/m}(C_h+B_h)$ is a global constant. Recall that $m$ is a constant that depends only on $\dim(G)$ (see proof of Lemma \ref{dimG}). Thus, by taking $D_1 = m$, we get the desired result.

\textbf{Step 2: Avoidance principle.} 
Let $H\subseteq S\subsetneq G$ be an intermediate orbit. 
First we shall fix a single closed $S$-orbit $Y=Sy$ with volume less that $R$, and show that there exists $D_2=D_2(\dim(G)) >0$ and $A_2=A_2(G/\Gamma, H) >0$ such that for all $D \ge D_2$ and $A \ge A_2$, the following is true: for all $t \ge A \log T$,
$$m_U(\{u\in B_1 ^U: \text{\rm{dist}}(a_{t}ux, Y) < R^{-D} \}) \ll R^{-1}.$$
Then, we will use Lemma \ref{finitely_many_intermediate_subgroups} and Corollary \ref{finiteness_closed_orbits_bounded_volume} to piece together the results for different choices of $Y$: we show that there exists $D_3=D_3(\dim(G)) > D_2$ such that for all $D \ge D_3$ and $A \ge A_2$, for all $t \ge A \log T$,
$$m_U(\{u\in B_1 ^U: \dist(a_{t}ux, Y) < R^{-D} \text{ for some } Y \in O_R \}) \ll R^{-1}$$
(here, $O_R = \{Y=Sy : H\subseteq S\subsetneq G, Y \text{ is closed}, \vol(Y)<R \}$ is the set of all closed orbits of volume less than $R$). 

\textbf{Step 2.1: Avoiding single closed orbit.}

Construct Margulis functions $f_Y$ and $F_Y$ with respect to $Y$ as in Section 6. If $I_Y(x)$ is empty, then $$f_Y(x)=h(x) \le T_x^{\delta_F} \le T^{\delta_F}.$$ Otherwise, we have 
$$f_Y(x)= \sum_{v\in I_Y(x)}\|v\|^{\delta_F}\le \dist(x,Y)^{-\delta_F}\#I_Y(x).$$

Since $x$ is $(R, 1/T)$--Diophantine with respect to $H$, $x$ is $(R, 1/T)$--Diophantine with respect to $Y$ and thus, $\dist(x,Y) \ge 1/T$. Combining with Proposition \ref{Prop:Number_of_Nearby_Sheets}, we have $$\dist(x,Y)^{-\delta_F}\#I_Y(x)\le \dist(x,Y)^{-\delta_F} \cdot C_1 \vol(Y)\le C_1T^{\delta_F} R.$$
    
Thus, using that $h(x)\le T^{\delta_F}$, we conclude
$$F_Y(x)=f_{Y}(x)+\lambda_1 \vol(Y) h(x) \le C_1T^{\delta_F} R +\lambda_1 T^{\delta_F} R = C_8 T^{\delta_F} R$$

where $C_8 := C_1 + \lambda_1$ is a global constant independent of $Y$.
    
Now take $A_2 = \delta_F t_F / \log 2$ (here, $t_F$ is as in Theorem \ref{Margulis_inequality_F}). If $t \ge A \log T$, then 
$$t \ge A \log T \ge A_2\log T_x \ge A_2\log h(x)^{1/\delta_F} \ge A_2/\delta_F \cdot \log 2 = t_F$$ 
and so by Corollary \ref{Cor:exponential_decay_F} we have
\begin{align*}
A_{t}(F_Y(x))&\le \frac{C_F}{2^{ t/t_F}}F_Y(x)+E_2 \vol(Y)\\
&\le \frac{C_F}{2^{ A \log T/t_F}} \cdot C_8 T^{\delta_F} R+E_2 R \\
& \le \left(\frac{1}{2^{A_2/t_F}}\right)^{\log T} \cdot C_FC_8 T^{\delta_F} R +E_2R\\ 
& = (C_FC_8+E_2) R.
\end{align*} 

If $\dist(x',Y) < R^{-D}$ for some $x'\in X$, then either there exists $v \in I_Y(x')$ with $||v|| < R^{-D}$ or $\varepsilon_h h(x')^{-\kappa} < R^{-D}$. In either case, we have that 
$$F_Y(x')=f_Y(x')+\lambda_1 \vol(Y) h(x') > \min \{R^{D\delta_F}, \lambda_1 \tau \varepsilon_h^{1/\kappa}R^{D/\kappa}\} \ge C_9 \cdot R^{D\delta_F}$$

where $C_9 := \lambda_1 \tau \varepsilon_h^{1/\kappa}$ is a global constant independent of $Y$. (Here, we are again using that $\delta_F := \min\{\delta_0/2, 1/\kappa\}$ and thus $\delta_F \le 1/\kappa$.)

Applying this observation to points $a_{t}ux$ and using Chebyshev's theorem we obtain,
\begin{align*}
&m_U(\{u\in B_1 ^U: \dist(a_{t}ux, Y) < R^{-D}\}) \\
& \le m_U(\{u\in B_1 ^U:F_Y(a_{t}ux) < C_9 \cdot R^{D\delta_F}\}) \\
& < C_9^{-1} R^{-D\delta_F}\cdot A_{t}(F_Y(x)) \\
&\le C_{10} R^{-(D\delta_F-1)}
\end{align*}
where $C_{10} := (C_FC_8+D_2)/C_9$ is another global constant independent of $Y$. Recall that $\delta_F$ is a global constant that depends only on $G$ and $H$ (see Remark \ref{dimG}). Take $D_2 = 2/\delta_F$.

\textbf{Step 2.2: Avoiding all closed orbits of small volume.} 

By Lemma \ref{finitely_many_intermediate_subgroups}, the number of intermediate subgroups $H\subseteq S\subsetneq G$ is finite; we shall denote this number as $N(G, H)$. By Theorem \ref{finiteness_closed_orbits_bounded_volume}, for each fixed $S$, the number of closed $S$-orbits $Y$ with $\vol(Y)<R$ is bounded by $C_6 R^{d_G / \delta_F}$. Therefore, the cardinality of the set 
$$ O_R = \{Y=Sy : H\subseteq S\subsetneq G, Y \text{ is closed}, \vol(Y)<R \}$$
is bounded above by $N(G, H) \cdot C_6 R^{d_G / \delta_F}$ and so
\begin{align*}
&m_U(\{u\in B_1 ^U: \dist(a_{t}ux, Y) < R^{-D} \text{ for some } Y \in O_R \})\\
& \le \sum_{Y \in O_R} m_U(\{u\in B_1^U: \text{\rm{dist}}(a_{t}ux, Y) < R^{-D}\})\\
& \le \sum_{Y \in O_R} C_9 R^{-(D\delta_F-1)} \\
& < N(G, H) \cdot C_6 R^{d_G/\delta_F} \cdot C_{10} R^{-(D\delta_F-1)} \\
& = N(G, H)C_6C_{10} \cdot R^{-D\delta_F+d_G/\delta_F+1}.
\end{align*}

Take $D_3 =D_3(\dim(G)):= (d_G/\delta_F+2)/\delta_F$ (so that $-D\delta_F+d_G/\delta_F+1 \ge -1$ for any $D \ge D_3$) and we get the desired avoidance principle: for all $D \ge D_3$ and $A \ge A_2$, for all $t \ge A\log T$, 
$$m_U(\{u\in B_1 ^U: \dist(a_{t}ux, Y) < R^{-D} \text{ for some } Y \in O_R) < C_{11} R^{-1}$$

(here, $C_{11} := N(G, H)C_6C_{10}$ is another global constant).

Now combine the results of Step 1 and Step 2. Take $D = D(\dim(G)) := \max\{D_1, D_3\}$, $A = A(G/\Gamma, H) := \max\{A_1, A_2\}$, and $C = C(G/\Gamma, H) := C_7 + C_{11}$, and we get 
\begin{align*}
& m_U(\{u\in B_1 ^U:a_{t}ux \text{ is not $(R, R^{-D}, R^{-D})$-Diophantine with respect to }H\}) \\
& \le m_U(\{u\in B_1 ^U: \inj(a_{t}ux) < R^{-D} \} \\
& \, \, \, \, \, \, \, + m_U(\{u\in B_1 ^U: \dist(a_{t}ux, Y) < R^{-D} \text{ for some } Y \in O_R) \\
& < C_7 R^{-1} + C_{11} R^{-1} = C R^{-1}.
\end{align*}

\end{proof}


 \end{document}